\documentclass[11pt]{article}
\usepackage{amsfonts}
\usepackage[latin9]{inputenc}
\usepackage{amsmath}
\usepackage{amssymb}
\usepackage{breqn}
\usepackage{url}
\usepackage{graphicx}
\usepackage[pdfstartview=FitH]{hyperref}
\usepackage{amsthm}
\usepackage{hyperref}
\usepackage{url}
\usepackage{enumitem}
\usepackage{verbatim}
\usepackage{eufrak}
\usepackage{cleveref}
\usepackage{array}
\usepackage{bm}

\crefformat{section}{\S#2#1#3} % see manual of cleveref, section 8.2.1
\crefformat{subsection}{\S#2#1#3}
\crefformat{subsubsection}{\S#2#1#3}

\usepackage{bbm}
\newcommand{\I}{\operatorname{I}}
\newcommand{\im}{\operatorname{Im}}

\newcommand{\Span}{\operatorname{span}}
\newcommand{\Aut}{\operatorname{Aut}}
\newcommand{\Sym}{\operatorname{Sym}}
\newcommand{\type}{\operatorname{type}}

\newcommand{\SC}{\operatorname{SC}}
\newcommand{\St}{\operatorname{St}}
\newcommand{\EL}{\operatorname{EL}}
\newcommand{\SL}{\operatorname{SL}}

\makeatletter

\begin{document}
\newtheorem{theorem}{Theorem}[section]
\newtheorem*{theorem*}{Theorem}
\newtheorem*{corollary*}{Corollary}
\newtheorem{lemma}[theorem]{Lemma}
\newtheorem{definition}[theorem]{Definition}
\newtheorem{claim}[theorem]{Claim}
\newtheorem{example}[theorem]{Example}
\newtheorem{remark}[theorem]{Remark}
\newtheorem{proposition}[theorem]{Proposition}
\newtheorem{corollary}[theorem]{Corollary}
\newtheorem{observation}[theorem]{Observation}
\newcommand{\subscript}[2]{$#1 _ #2$}
\author{
Tali Kaufman
\footnote{Department of Computer Science, Bar-Ilan University, kaufmant@mit.edu, research supported by ERC and BSF.}
\and
Izhar Oppenheim
\footnote{Department of Mathematics, Ben-Gurion University of the Negev, Be'er Sheva 84105, Israel, izharo@bgu.ac.il, research supported by ISF.}
}

\title{High dimensional expanders and coset geometries}
\maketitle

\begin{abstract}

High dimensional expanders is a vibrant emerging field of study. Nevertheless, the only known construction of bounded degree high dimensional expanders is based on Ramanujan complexes, whereas one dimensional bounded degree expanders are abundant.

In this work, we construct new families of bounded degree high dimensional expanders obeying the local spectral expansion property. This property has a number of important consequences, including geometric overlapping, fast mixing of high dimensional random walks, agreement testing and agreement expansion. Our construction also yields new families of expander graphs which are close to the Ramanujan bound, i.e., their spectral gap is close to optimal. 

The construction is quite elementary and it is presented in a self contained manner; This is in contrary to the highly involved previously known construction of the Ramanujan complexes. The construction is also very symmetric (such symmetry properties are not known for Ramanujan complexes) ; The symmetry of the construction could be used, for example, in order to obtain good symmetric LDPC codes that were previously based on Ramanujan graphs.

The main tool that we use for is the theory of coset geometries.
Coset geometries arose as a tool for studying finite simple groups. Here, we show that coset geometries arise in a very natural manner for groups of elementary matrices over any finitely generated algebra over a commutative unital ring. In other words, we show that such groups act simply transitively on the top dimensional face of a pure, partite, clique complex. 

\end{abstract}

\section{Introduction}

The main motivation for this paper is the construction of new examples of high dimensional expanders (which are simplicial complexes with expansion properties that serve as an analogue of expansion in graphs). 

In \cite{KOStoc}, we gave such a construction in a rather ad-hoc manner. When we axiomatized and generalized the ideas, we understood two things: First, that we essentially rediscovered some basic facts of the well-established theory of coset geometries. Second, that we found new examples of coset geometries for elementary groups and Steinberg groups over rings that were yet unknown and are of independent interest. These new examples include even new coset geometries for well-studied groups - for instance, we show that for $n \geq 2$ there is a pure $n$-dimensional simplicial (clique) complex $X$ such that the group $SL_{n+1} (\mathbb{Z})$ acts simply transitively on top dimensional faces of $X$.    

\subsection{High dimensional expanders}

In recent years there has been a surge of activity regarding the new topic of high dimensional expanders. There is currently no single definition for what a high dimensional expander is: there are several (non equivalent) definitions that match different applications. There are however motivating examples of simplicial complexes that should be high dimensional expanders under any reasonable definition: these are the Ramanujan complexes constructed in \cite{LSV1}, \cite{LSV2} that have many desirable properties. Currently, the only known construction of bounded degree high dimensional expanders, according to either of the potential definitions, is based on Ramanujan complexes, whereas one dimensional bounded degree expanders are abundant.

The construction of Ramanujan complexes is far from being simple in the sense that describing it requires the Bruhat-Tits buildings and the explicit construction in \cite{LSV2} is not straightforward. This highly involved construction makes it less accessible to potential ``end-users'' of high dimensional expanders from various areas of mathematics and computer science.

In this paper, we construct new families of bounded degree high dimensional expanders obeying the (one-sided) local spectral expansion property.  To explain this notion, we recall that given an $n$-dimensional simplicial complex $X$ and a simplex $\tau$ in $X$, the {\em link} of $\tau$ denoted $X_{\tau}$ is the complex obtained by taking all faces in $X$ that contain $\tau$ and removing $\tau$ from them. Thus, if $\tau$ is of dimension $i$ (i.e. $\tau \in X(i)$ ) then $X_{\tau}$ is of dimension $n-i-1$.

For every $-1 \leq i \leq n-2$, the one skeleton of $X_{\tau}$ is a graph with a weight function on the edges (see exact definition in \cref{Construction of HD exp sec.}). Denote the second largest eigenvalue of the random walk on $X_\tau$ with respect to the weight on the edges by $\mu_{\tau}$.

\begin{definition}[one-sided local spectral expander]
A pure $n$-dimensional complex $X$ is a {\em one-sided $\lambda$-local-spectral expander} if for every $-1 \leq i \leq n-2$, and for every $\tau \in X(i)$,  $\mu_{\tau} \leq \lambda$.
\end{definition}

This property is very desirable because it was shown by the second named author in \cite{OppLocII} that for partite complexes, local spectral expansion implies mixing and geometrical overlapping. Moreover, in \cite{KM} it was shown that one-sided local spectral expansion implies fast mixing of high order random walks. In \cite{DK} it was shown that {\em two-sided} local spectral expansion property (which is stronger than the one-sided version) implies {\em optimal} mixing of high order random walks, and this in turn is useful for obtaining PCP agreement tests, direct product testing and direct sum testing based on complexes with the {\em two-sided} local spectral expansion property.
In \cite{KO}, the authors showed that {\em one-sided} local spectral expansion implies optimal mixing of the high dimensional random walk, and this in turn implies agreement expansion using \cite{DK} argument. 

We provide the following elementary construction of one-sided local spectral expanders based on more general ideas explained in \cref{Coset complexes for elementary matrix groups - intro subsec} below.

Recall that given an algebra $R$, the group $EL_{n+1} (R)$ is the group of elementary matrices with coefficients in $R$: for $1 \leq i, j \leq n+1, i \neq j$ and $r \in R$, let $e_{i,j} (r)$ be the $(n+1) \times (n+1)$ matrix with $1$'s along the main diagonal, $r$ in the $(i,j)$ entry and $0$'s in all the other entries. $EL_{n+1} (R)$ are the matrices generated by $\lbrace e_{i,j} (r) : 1 \leq i, j \leq n+1, i \neq j, r \in R \rbrace$. 

Let $n \geq 2$ and let $q$ be a prime power. For $s \in \mathbb{N}, s>n$, define $G^{(s)} = \EL_{n+1} (\mathbb{F}_q / \langle t^s \rangle)$ and let $K_{\lbrace i \rbrace}, i=0,...n$ be the following subgroups of $G^{(s)}$:
$$K_{\lbrace i \rbrace} = \langle e_{j, j+1} (a+bt) : j \in \lbrace 0,...,n\rbrace \setminus \lbrace i \rbrace, a,b \in \mathbb{F}_q \rangle,$$
where $j+1$ is taken modulo $n+1$. 

Define $X^{(s)}$ to be the following simplicial complex:
\begin{itemize}
\item The vertices of $X^{(s)}$ are the cosets $\bigcup_{i=0}^n \lbrace g K_{\lbrace i \rbrace} : g \in G^{(s)} \rbrace$. 
\item Two vertices $g K_{\lbrace i \rbrace}$ and $g' K_{\lbrace i' \rbrace}$ are connected by an edge if $g K_{\lbrace i \rbrace} \cap g' K_{\lbrace i' \rbrace} \neq \emptyset$ and $i \neq i'$.
\item If the vertices $g_0 K_{\lbrace i_0 \rbrace},...,g_k K_{\lbrace i_k \rbrace}$ are pairwise connected by edges, then $\lbrace g_0 K_{\lbrace i_0 \rbrace},...,g_k K_{\lbrace i_k \rbrace} \rbrace$ is a $k$-dimensional simplex in $X^{(s)}$, i.e., $X^{(s)}$ is the clique complex of its one-skeleton.
\end{itemize}

\begin{theorem}
Let $n \geq 2$ and let $q$ be a prime power such that $q > (n-1)^2$. Then for every $s > n$, the following holds for $X^{(s)}$:
\begin{enumerate}
\item $X^{(s)}$ is a pure $n$-dimensional, $(n+1)$-partite, strongly gallery connected clique complex with no free faces.
\item $X^{(s)}$ is finite and the number of vertices of $X^{(s)}$ tends to infinity as $s$ tends to infinity.
\item There is a constant $Q=Q(q)$ such that for every $s$, each vertex of $X^{(s)}$ is contained in exactly $Q$ $n$-dimensional simplices.
\item Each $1$-dimensional link of $X^{(s)}$ has a spectral gap that is greater or equal to $1-\frac{1}{\sqrt{q}}$.
\item $X^{(s)}$ is a $\frac{1}{\sqrt{q}-(n-1)}$-local spectral expander. In particular, the spectral gap of the one-skeleton of $X^{(s)}$ is greater or equal to $1-\frac{1}{\sqrt{q}-(n-1)}$.
\item The automorphisms group of $X^{(s)}$ contains a subgroup that acts transitively on $n$-dimensional simplices and rotates / inverts the types. 
\end{enumerate}
Thus, given any $\lambda >0$, we can take $q$ large enough such that $\lbrace X^{(s)} : s >n \rbrace $ will be an infinite family of $n$-dimensional $\lambda$-local spectral expanders which are uniformly locally finite.
\end{theorem}  

We note that our constructions are very symmetric, i.e., the group of simplicial automorphisms of the complexes we construct is shown to be large  (see Corollary \ref{transitive action coro 1} or Corollary \ref{transitive oriented action coro 1} for exact formulation). In this respect, our construction surpasses the symmetry properties of constructions of Ramanujan complexes (see more on that below). In the two dimensional case, this symmetry is particulary stricking, since the stablizer of any vertex acts transitevely on the link of the vertex. This could be used, for example, to obtain good symmetric LDPC codes (see \cite{KW}) that were previously based on Ramanujan graphs \cite{KaufmanL}. We further note that the $1$-skeletons of local spectral expanders are expander graphs, and therefore, as a byproduct, our construction also yields new families of expander graphs.

%We believe that the new constructions of bounded degree high dimensional expanders that we present in this work could contribute to the general study of the phenomenon of high dimensional expansion. Prior to out work, essentially there was only one known construction of a family of bounded degree high dimensional expanders and that construction was highly non-elementary.  Our work provides many new examples of bounded degree high dimensional expanders. These new constructions could shed new light of the study of what high dimensional expansion should really mean, and how it can be used in various computer science and mathematical applications.

\textbf{Comparison to Ramanujan complexes.} It is interesting to compare our construction to the previously known construction of Ramanujan complexes. Ramanujan complexes exhibit a better ratio between the spectral bound and the the degree bound, but our constructions seem much more symmetric: Ramanujan complex is constructed by choosing a lattice $\Lambda <  PGL_{n+1} (F)$ and passing to the quotient 
$\Lambda \backslash X = \Lambda \backslash PGL_{n+1} (F) /K$ of the Euclidean building. 

We note that $PGL_{n+1} (F)$ can not act on $\Lambda \backslash PGL_{n+1} (F) /K$ and therefore the transitivity properties of this action on $X$ are a-priori lost when passing to the quotient. In \cite{LSV1}, the authors discuss choosing a lattice for which the quotient does admit some transitivity properties, but our construction is still much more symmetric. 

Also, Ramanujan complexes are known to be cosystolic expanders (see definition and proof in \cite{EK}), while this is yet unknown for our constructions. 

\subsection{Coset geometries and coset complexes}

The systematic study of coset geometries and diagram geometries was initiated by Buekenhout \cite{B1} with the main motivation of studying (and classifying) finite simple groups. The idea was to generalize Tits' axioms for buildings in order to associate incidence geometries to simple groups.
%The goal of this study was, in the words of Solomon \cite{Solomon}, to find \textit{``elegant set of axioms, in the spirit of Tits' axioms for a building, defining a class of geometries for all the finite simple groups and perhaps more, but not too much more''}. This approach was indeed very fruitful for studying finite simple groups (see the survey in \cite{BP}), although in the end, as noted in \cite{DiaGeomBook}: \textit{``the classification of finite simple groups has been completed without a satisfactory
%framework offered by diagram geometry''}.

We are interested in the geometric interpretation of coset geometries as simplicial complexes, which we call coset complexes. Below, we will only use the terminology of simplicial complexes and not the terminology of incidence systems that is usually used to describe coset geometries in the standard literature. We note that these two terminologies (simplicial complexes / incidence systems) are equivalent and in Appendix \ref{Incidence systems appen} we provide a dictionary between these terminologies. 

Given a group $G$ with subgroups $K_{\lbrace 0 \rbrace},...,K_{\lbrace n \rbrace}$, consider the graph with vertices $V= \bigcup_{i} \lbrace g K_{\lbrace i \rbrace} : g \in G \rbrace$ and such that there is an edge connecting $ g K_{\lbrace i \rbrace}$ and $g' K_{\lbrace j \rbrace}$ if and only if $i \neq j$ and $g K_{\lbrace i \rbrace} \cap g' K_{\lbrace j \rbrace} \neq \emptyset$ (in the literature on coset geometry, this graph is considered as an incidence system). We define the \textit{coset complex} $X=X(G, (K_{\lbrace i \rbrace})_{i \in \lbrace 0,...,n\rbrace})$ to be the clique complex over this graph, i.e., $\lbrace g_0 K_{\lbrace i_0 \rbrace},...,g_k K_{\lbrace i_k \rbrace} \rbrace$ spans a $k$-simplex in $X$ if and only if for every $0 \leq j,j' \leq k$, $j \neq j'$, there is an edge connecting $g_j K_{\lbrace i_j \rbrace}$ and $g_{j'} K_{\lbrace i_{j'} \rbrace}$. 

With the above definition of $X$, it is not surprising that the geometric properties of $X$ reflect algebraic properties of $G, K_{\lbrace 0 \rbrace},...,K_{\lbrace n \rbrace}$ and vice versa. In this article, we will work with very modest demands on \\ $(G, (K_{\lbrace i \rbrace})_{i \in \lbrace 0,...,n\rbrace})$. Namely, we give three axioms $(\mathcal{A}1)-(\mathcal{A}3)$ on \\ $(G, (K_{\lbrace i \rbrace})_{i \in \lbrace 0,...,n\rbrace})$ and we call $(G, (K_{\lbrace i \rbrace})_{i \in \lbrace 0,...,n\rbrace})$ a \textit{subgroup geometry system} if it satisfies these axioms (see exact details in Definition \ref{SGS def}). 

Using the theory of coset geometry (see Theorem \ref{cost geom thm}) we derive that if $(G, (K_{\lbrace i \rbrace})_{i \in \lbrace 0,...,n\rbrace})$ is a subgroup geometry system, then \\
 $X=X(G, (K_{\lbrace i \rbrace})_{i \in \lbrace 0,...,n\rbrace})$ has several desirable properties:
\begin{theorem}{\cite[Section 1.8]{DiaGeomBook}}
\label{cost geom thm - intro}
Let $n \in \mathbb{N}$ and $G$ be group with subgroups $K_{\lbrace i \rbrace}, i \in \lbrace 0,...,n\rbrace$. If $(G, (K_{\lbrace i \rbrace})_{i \in \lbrace 0,...,n\rbrace})$ is a subgroup geometry system, then $X=X(G, (K_{\lbrace i \rbrace})_{i \in \lbrace 0,...,n\rbrace})$ is a pure $n$-dimensional, $(n+1)$-partite, strongly gallery connected clique complex with no free faces (see \cref{Simplicial complexes subsec} for the terminology regarding simplicial complexes). Furthermore, $G$ acts on $X$ by type preserving automorphisms and this action is transitive on $n$-dimensional simplices of $X$.
\end{theorem}

We further explore the connections between $X$ and the subgroup structure of $G$ in several directions that, as far as we know, were not addressed in the study of coset geometries:
\begin{enumerate}
\item We show if there is a group $\Gamma < \Aut (G)$ that preserve the subgroups $K_{\lbrace 0 \rbrace},...,K_{\lbrace n \rbrace}$, then $\Gamma$ acts on $X$ simplicially (see exact formulation in \cref{Symmetries of the constructions sect}).
\item We observe that for $N \triangleleft G$, if $N \cap K_{\lbrace i \rbrace}$ is trivial for every $i$, then $G / N$ also have a subgroup geometry system $(G/N, (K_{\lbrace i \rbrace} N / N)_{i \in \lbrace 0,...,n\rbrace})$ and $X$ is a covering of the $X(G/N, (K_{\lbrace i \rbrace} N / N)_{i \in \lbrace 0,...,n\rbrace}))$ (see exact formulation if \cref{Passing to quotients subsect}).
\item We show that in cases where for every $0 \leq i,j \leq n$, $i \neq j$, the subgroups $K_{\lbrace 0,...,n\rbrace \setminus \lbrace i,j \rbrace}$ are finite, the spectrum of the $1$-dimensional links of $X$ can be bounded using the representation theory of these groups (see exact formulation in Appendix \ref{spectral gap appen}).
\end{enumerate}

\begin{remark}
In \cite{LLR}, Lubotzky, Luria and Rosenthal constructed what they called Aboreal complexes. Although their terminology differs from ours, their construction can be defined using coset complexes as follows: for $l \in \mathbb{N}$, $l \geq 2$, let $C_l$ be the cyclic group with $l$ elements. Let $G_{n,l}$ be the free product of $n+1$ copies of $C_l$, i.e., let $C_l^i$ be a cyclic group indexed by $i =0,...,n$ and $G_{n,l} = \ast_{i=0}^n C_l^i$. For $i \in \lbrace 0,...,n \rbrace$, take $K_{\lbrace i \rbrace} = \ast_{j \neq i} C_l^j$. Then $(G_{n,l}, ( K_{\lbrace i \rbrace}  )_{i \in \lbrace 0,...,n \rbrace} )$ is a subgroup geometry system and $X(G_{n,l}, ( K_{\lbrace i \rbrace}  )_{i \in \lbrace 0,...,n \rbrace} )$ is the Aboreal complex denoted $T_{n,l}$ in \cite{LLR}.
\end{remark}

\subsection{Coset complexes for elementary matrix groups}
\label{Coset complexes for elementary matrix groups - intro subsec}
Let $R$ be a unital commutative ring and $\mathcal{R}$ be a finitely generated $R$-algebra. Let $\lbrace t_1,...,t_l\rbrace$ be a generating set of $\mathcal{R}$ and let $T$ be the $R$-module generated by $\lbrace 1, t_1,...,t_l\rbrace$, i.e., $T= \lbrace r_0+ \sum_i r_i t_i : r_i \in R \rbrace$. 

For $n \in \mathbb{N}, n \geq 1$, the \textit{group of elementary $(n+1) \times (n+1)$ matrices}, denoted $\EL_{n+1} (\mathcal{R})$, is the group $\EL_{n+1} (\mathcal{R}) = \langle e_{i,j} (r) : 0 \leq i, j \leq n, i \neq j, r \in \mathcal{R} \rangle$, where $e_{i,j} (r)$ is the  matrix with $1$'s along the main diagonal, $r$ in the $(i,j)$ entry and $0$'s in all the other entries. Note that we index the rows and columns of the matrices by $0,...,n$. 

For this group, we define subgroups $K_{\lbrace i \rbrace} < \EL_{n+1} (\mathcal{R})$, where $i \in \lbrace 0,...,n\rbrace$ by
$$K_{\lbrace i \rbrace} = \langle e_{j, j+1} (m) : j \in \lbrace 0,...,n\rbrace \setminus \lbrace i \rbrace, m  \in T \rangle,$$
where $j+1$ is taken modulo $n+1$, i.e., for $j=n$, $e_{n,n+1} (m) = e_{n,0} (m)$. 

\begin{theorem}
\label{geometry system for elementary mat groups thm - intro}
For any choice of $R,\mathcal{R},T$ as above, \\ $(\EL_{n+1} (\mathcal{R}), (K_{\lbrace i \rbrace})_{i \in \lbrace 0,...,n\rbrace})$ is a subgroup geometry system and therefore the coset complex $X(\EL_{n+1} (\mathcal{R}), (K_{\lbrace i \rbrace})_{i \in \lbrace 0,...,n\rbrace})$ has the properties specified in Theorem \ref{cost geom thm - intro} above.

Furthermore, the dihedral group $D_{n+1}$ acts by automorphisms on \\ $\EL_{n+1} (\mathcal{R})$ and preserve the subgroups $K_{\lbrace 0 \rbrace},...,K_{\lbrace n \rbrace}$. Therefore $D_{n+1}$ also acts simplicially on $X(\EL_{n+1} (\mathcal{R}), (K_{\lbrace i \rbrace})_{i \in \lbrace 0,...,n\rbrace})$ (by type rotating and type inverting automorphisms). 
\end{theorem}

We also show that all of the above also applies to the Steinberg group $\St_{n+1} (\mathcal{R})$ and that $X(\St_{n+1} (\mathcal{R}), (K_{\lbrace i \rbrace})_{i \in \lbrace 0,...,n\rbrace})$ is a covering of \\
$X(\EL_{n+1} (\mathcal{R}), (K_{\lbrace i \rbrace})_{i \in \lbrace 0,...,n\rbrace})$.

Above, we stated that this general construction yields new examples of high dimensional expanders. Here we will like to outline another example which seems fundamental: let $n \geq 2$ and $R = \mathcal{R} = \mathbb{Z}$ with the generating set $\lbrace 1 \rbrace$. In this case $T = \mathbb{Z}$ and $\EL_{n+1} (\mathbb{Z}) = \SL_{n+1} (\mathbb{Z})$ (the equality $\EL_{n+1} (R) = \SL_{n+1} (R)$ is known when $R$ is a commutative Euclidean ring and in some other cases - see \cite[4.3B]{HOMBook}). With the subgroups $K_{\lbrace i \rbrace}$ as above, $X(\SL_{n+1} (\mathbb{Z}), (K_{\lbrace i \rbrace})_{i \in \lbrace 0,...,n\rbrace})$ is a locally infinite simplicial complex on which $ \SL_{n+1} (\mathbb{Z})$ acts simplicially and the action is simply transitive on $n$-dimensional simplices.

\subsection{Organization} The paper is organized as follows. In \cref{General construction sect}, we explain the basic construction of coset complexes, based on the general theory of coset geometries and then further explore this construction in two directions: passing to quotients and expanding the automorphism group of coset complexes. In \cref{Elementary matrices construction sect}, we show how to construct coset complexes based on elementary matrix groups and Steinberg groups over ring algebras and explore the symmetries of these constructions. In \cref{Construction of HD exp sec.}, we show how to use the construction of coset complexes based on based on elementary matrix groups to produce new examples of local spectral high dimensional expanders. In order to ease the reading several technical discussions were added as appendices. In \cref{Incidence systems appen}, we provide a dictionary between the terminology of coset complexes that we use and the (equivalent) terminology of coset geometries that is used in some of the literature. In \cref{spectral gap appen}, we show how to use representation theory to bound the spectrum of $1$-dimensional links in coset complexes (this result was known to experts, but we could not find a good reference for it).

\section{Coset geometries and simplicial complexes}
\label{General construction sect}

In this section, we will explain how coset geometries give rise to simplicial complexes. 

The literature on coset geometries usually uses the terminology of incidence geometry and not the terminology of simplicial complexes (although the two terminologies are equivalent in the case of partite complexes). We will use only the terminology of simplicial complexes (even when referring to literature that uses incidence geometry terminology) and Appendix \ref{Incidence systems appen} provides a dictionary between the two terminologies.

%In this section, we will use ideas taken from the theory of complexes of groups (see Bridson and Haefliger's book \cite{BridsonH}[Part 2, Chapter 12 and Appendix C]) in order to describe a general construction of complexes that arise from groups with a specific subgroup structure (see below).
%Our construction can be seen as a special case of (developable) simple complexes of groups. More general cases are explained in \cite{BridsonH}, but our emphsis is different than the one in \cite{BridsonH}: in \cite{BridsonH} the main question that is studied is given a complex of groups is it developable and what is the fundamental group of it (if it exists). Our point of view is opposite: we start with a group $G$ and describe a subgroup structure that gives rise to a simplicial complex on which $G$ acts simplicially. This type of construction can be seen as a special case of previous constructions (see \cref{Connections with previous constructions subsec}).

\subsection{Simplicial complexes}
\label{Simplicial complexes subsec}
Let $X$ be an $n$-dimensional simplicial complex over a vertex set $V$. We set some notation and terminology. 

For every $-1 \leq k \leq n$, we denote by $X(k)$ the set of $k$-dimensional simplices in $X$, i.e., $\sigma \in X(k)$ is a set of vertices of $X$ of cardinality $k+1$ that form a $k$-dimensional simplex in $X$ (we use the convention $X(-1) =  \lbrace \emptyset \rbrace$). Note that $X(0)$ is not exactly $V$: $v \in V \Leftrightarrow \lbrace v \rbrace \in X(0)$. 

$X$ is called \textit{pure $n$-dimensional} if every simplex of $X$ is contained in an $n$-dimensional simplex. The simplicial complex $X$ is called \textit{$(n+1)$-partite}, if there are disjoint sets $S_0,...,S_n \subset V$, called \textit{the sides of $X$}, such that for  $V = \bigcup S_i$ and for every $\sigma \in X(n)$ and every $0 \leq i \leq n$, $\vert \sigma \cap V_i \vert = 1$, i.e., every $n$-dimensional simplex has exactly one vertex in each of the sides of $X$.  In a pure $n$-dimensional, $(n+1)$-partite complex $X$, each simplex $\sigma \in X(k)$ has a \textit{type} which is a subset of $\lbrace 0,..., n\rbrace$ of cardinality $k+1$ that is defined by $\type (\sigma) = \lbrace i : \sigma \cap S_i \neq \emptyset \rbrace$. A pure $n$-dimensional simplicial complex $X$ is said to have a \textit{free face}, if there is $\tau \in X(n-1)$ such that there is a unique $\sigma \in X(n)$ such that $\tau \subseteq \sigma$.Thus, $X$ has no free faces if every simplex of dimension strictly less than $n$ is contained in at least two $n$-dimensional simplices.

The \textit{$1$-skeleton} of $X$ is the graph with vertices $X$ and edges $X(1)$. A simplicial complex $X$ is said to be \textit{connected} if the $1$-skeleton of $X$ is connected as a graph. A simplicial complex $X$ is called a \textit{clique complex} (or a \textit{flag complex}) if every clique in the $1$-skeleton of $X$ spans a simplex, i.e., if $v_0,...,v_k \in V$ such that for every $0 \leq i < j \leq k$, $\lbrace v_i, v_j \rbrace \in X(1)$, then $\lbrace v_0,...,v_k \rbrace \in X(k)$.

Given a simplex $\tau \in X(k)$, the \textit{link of $\tau$} is the subcomplex of $X$, denoted by $X_\tau$, and defined by
$$X_\tau = \lbrace \eta \in X : \tau \cap \eta = \emptyset, \tau \cup \eta \in X \rbrace.$$
We note that if $X$ is pure $n$-dimensional and $(n+1)$-partite, then for every $-1 \leq k \leq n$ and every $\tau \in X(k)$, $X_\tau$ is pure $(n-k-1)$-dimensional and $(n-k)$-partite. A simplicial complex $X$ is called \textit{locally finite} if for every $\tau \in X$, $\tau \neq \emptyset$, $X_\tau$ is finite, i.e., $\vert X_\tau (0) \vert < \infty$.

A \textit{gallery} in $X$ is a sequence of top dimensional simplices $\sigma_1,...,\sigma_l \in X(n)$ such that for every $i$, $\sigma_{i} \cap \sigma_{i+1} \in X (n-1)$.  A simplicial complex $X$ is called \textit{gallery connected} if for every $\sigma, \sigma' \in X(n)$ there is a gallery connecting $\sigma$ and $\sigma '$, i.e., there are $\sigma_0,...,\sigma_l \in X(n)$ such that $\sigma_0 = \tau, \sigma_l = \tau'$ and for every $i$, $\sigma_{i} \cap \sigma_{i+1} \in X (n-1)$. A simplicial complex $X$ will be called \textit{strongly gallery connected} if it is gallery connected and for every $0 \leq k \leq n-2$, $\tau \in X(k)$, $X_\tau$ is gallery connected (as a complex of dimension $n-k-1$).

\begin{remark}
\label{strongly gallery connected to locally connected rmk}
A simple exercise shows that $X$ is \textit{strongly gallery connected}  if and only if it is connected (i.e., the $1$-skeleton of $X$ is connected) and for every $0 \leq k \leq n-2$, $\tau \in X(k)$, $X_\tau$ is connected (i.e., the $1$-skeleton of $X_\tau$ is connected). 
\end{remark}

For a pure $n$-dimensional complex $X$, a simplex $\tau$ of dimension $<n$ is said to be a \textit{free face} of $X$, if it is contained in exactly one $n$-dimensional simplex of $X$. 
%In a gallery connected complex $X$, the \textit{gallery distance} between $\sigma, \sigma '$ is the length of shortest gallery connecting $\sigma$ and $\sigma '$. 

\subsection{Coset complexes}

Below, we state known results regarding coset geometries in the language of simplicial complexes. We note that in all the references given below, the results are stated in the terminology of incidence geometries and refer the reader to Appendix \ref{Incidence systems appen} for a dictionary between these terminologies.

Let $G$ be a group with subgroups $K_{\lbrace i \rbrace}, i \in \I$, where $\I$ is a finite set. The \textit{coset complex} $X=X(G, (K_{\lbrace i \rbrace})_{i \in \I})$ is a simplicial complex defined as follows: 
\begin{enumerate}
\item The vertex set of $X$ is composed of disjoint sets $S_i = \lbrace g K_{\lbrace i \rbrace} : g \in G \rbrace$.
\item For two vertices $g K_{\lbrace i \rbrace}, g' K_{\lbrace j \rbrace}$ where $i,j \in \I, g,g' \in G$, $\lbrace g K_{\lbrace i \rbrace}, g' K_{\lbrace j \rbrace} \rbrace \in X(1)$ if $i \neq j$ and  $g K_{\lbrace i \rbrace} \cap g' K_{\lbrace j \rbrace} \neq \emptyset$.
\item The simplicial complex $X$ is the clique complex spanned by the $1$-skeleton defined above, i.e., $\lbrace g_0 K_{\lbrace i_0 \rbrace},..., g_k  K_{\lbrace i_k \rbrace} \rbrace \in X(k)$ if for every $0 \leq j,j' \leq k$, $g_j K_{\lbrace i_j \rbrace} \cap g_{j'} K_{\lbrace i_{j'} \rbrace} \neq \emptyset$.
\end{enumerate}
With the above definition, $X(G, (K_{\lbrace i \rbrace})_{i \in \I})$ is clearly a $\vert \I \vert$-partite clique complex. 

We define an action of $G$ on the vertices of $X(G, (K_{\lbrace i \rbrace})_{i \in \I})$: for every $g, g' \in G$, $i \in \I$, $g. (g' K_{\lbrace i \rbrace}) = g g' K_{\lbrace i \rbrace}$. It is not hard to verify that this action is simplicial, i.e., for every $\lbrace g_0 K_{\lbrace i_0 \rbrace},..., g_k  K_{\lbrace i_0 \rbrace} \rbrace \in X(G, (K_{\lbrace i \rbrace})_{i \in \I}) (k)$ and every $g \in G$,
$$g. \lbrace g_0 K_{\lbrace i_0 \rbrace},..., g_k  K_{\lbrace i_0 \rbrace} \rbrace = \lbrace g g_0 K_{\lbrace i_0 \rbrace},..., g g_k  K_{\lbrace i_0 \rbrace} \rbrace \in X(G, (K_{\lbrace i \rbrace})_{i \in \I}) (k).$$
We also note that this action is type preserving, i.e., for every simplex $\sigma$ in $X(G, (K_{\lbrace i \rbrace})_{i \in \I})$ and every $g \in G$, $\type (\sigma) = \type (g.\sigma)$ 

For a group $G$ and subgroups $K_{\lbrace i \rbrace}, i \in \I$, we denote for every $\tau \subseteq \I, \tau \neq \emptyset$, $K_\tau = \bigcap_{i \in \tau} K_{\lbrace i \rbrace}$ and further denote $K_{\emptyset} = G$. 
\begin{definition}
\label{SGS def}
Let $n \in \mathbb{N}$, $\I$ a finite set of cardinality $n+1$ and $G$ be group with subgroups $K_{\lbrace i \rbrace}, i \in \I$. We call $(G, (K_{\lbrace i \rbrace})_{i \in \I})$ a subgroup geometry system if:
\begin{enumerate}[label=($\mathcal{A}${{\arabic*}})]
\item For every $\tau, \tau' \subseteq \I$, $K_{\tau \cap \tau '} = \langle K_{\tau},  K_{\tau '} \rangle$, where $\langle K_\tau, K_\tau' \rangle$ denotes the subgroup of $G$ generated by $K_\tau$ and $K_{\tau'}$.
\item For every $\tau \subsetneqq \I$ and $i \in \I \setminus \tau$, $K_{\tau} K_{\lbrace i \rbrace} = \bigcap_{j \in \tau} K_{\lbrace j \rbrace} K_{\lbrace i \rbrace}$.
\item For every $i \in \I$, $K_\I \neq K_{\I \setminus \lbrace i \rbrace}$.
\end{enumerate}
\end{definition}

\begin{remark}
The reason for the name subgroup geometry system is that in the terminology of incidence systems, the coset incidence system that fulfills conditions $(\mathcal{A} 1)- (\mathcal{A} 3)$ is an $\I$-geometry (see \cite{DiaGeomBook}).
\end{remark}

After this definition, we state the following theorem from the introduction (Theorem \ref{cost geom thm - intro}):
\begin{theorem}
\label{cost geom thm}
Let $n \in \mathbb{N}$, $\I$ a finite set of cardinality $n+1$ and $G$ be group with subgroups $K_{\lbrace i \rbrace}, i \in \I$. If $(G, (K_{\lbrace i \rbrace})_{i \in \I})$ is a subgroup geometry system, then $X=X(G, (K_{\lbrace i \rbrace})_{i \in \I})$ is a pure $n$-dimensional, $(n+1)$-partite, strongly gallery connected clique complex with no free faces. Furthermore, $G$ acts on $X$ by type preserving automorphisms and this action is transitive on $n$-dimensional simplices of $X$, i.e., for every $\sigma, \sigma ' \in X(n)$ there is $g \in G$ such that $g. \sigma = \sigma '$. 
\end{theorem} 

\begin{proof}
The proof of this Theorem is the consequence of \cite[Section 1.8]{DiaGeomBook} (see page 36 and the discussion that precedes it). 
\end{proof}

\begin{remark}
The other direction for this Theorem is also true: let $X$ be an $(n+1)$-partite, strongly gallery connected clique complex with no free faces and  $G$ be a group that acts on $X$ by type preserving automorphisms such that the action is transitive on $n$-dimensional simplices of $X$. Fix $I= \lbrace v_0,...,v_n \rbrace \in X(n)$ and for $v_i \in I$, denote $K_{\lbrace i \rbrace} = \lbrace g \in G : g. v_i = v_i \rbrace$. Fix as above $K_\emptyset = G$ and for every $\tau \subsetneqq I$, fix $K_\tau = \bigcap_{i \in \tau} K_{\lbrace i \rbrace}$. We observe that for every $\tau \subseteq I$, $K_{\tau} = \lbrace g \in G : g. \tau = \tau \rbrace$. Then with these notations, $(G, (K_{\lbrace i \rbrace})_{i \in \I})$ fulfill conditions $(\mathcal{A} 1)- (\mathcal{A} 3)$ and in that case $X$ is isomorphic to $X(G, (K_{\lbrace i \rbrace})_{i \in \I})$.
The proof of these facts is left for the reader and we will not make any use of these facts in the sequel.
\end{remark}

We note that for a group $G$ with subgroups $K_{\lbrace i \rbrace}, i \in \I$, such that $(G, (K_{\lbrace i \rbrace})_{i \in \I})$ is a subgroup geometry system, the complex $X(G, (K_{\lbrace i \rbrace})_{i \in \I})$ can be given a different description using cosets of the form $g K_{\tau}$:  
\begin{proposition}
\label{different descrip prop}
Let $n \in \mathbb{N}$, $\I$ a set of cardinality $n+1$ and $(G, (K_{\lbrace i \rbrace})_{i \in \I})$ a subgroup geometry system. Denote $X = X(G, (K_{\lbrace i \rbrace})_{i \in \I})$. 
For every $-1 \leq k \leq n$, denote $\I (k) = \lbrace \tau \subseteq \I : \vert \tau \vert = k+1 \rbrace$ and define the function
$$f: X (k) \rightarrow \lbrace  g K_\tau : g \in G, \tau \in \I (k) \rbrace,$$
as
$$f(\lbrace g K_{\lbrace i \rbrace} : i \in \tau \rbrace) =  g K_\tau.$$
Then $f$ is well defined, bijective and equivariant under the action of $G$, i.e., for every $g' \in G$
$$f(g' . \lbrace g K_{\lbrace i \rbrace} : i \in \tau \rbrace ) =  g' g K_\tau.$$
Furthermore, for every $\tau, \tau' \in \bigcup_{-1 \leq k \leq n} \I (k)$ and every $g, g' \in G$, $f^{-1} (g K_\tau) \subseteq f^{-1} (g' K_{\tau '})$ if and only if $\tau \subseteq \tau'$ and $(g')^{-1} g \in K_\tau$.
\end{proposition}

\begin{proof}
Without loss of generality, we will assume $\I = \lbrace 0,...,n \rbrace$. By Theorem \ref{cost geom thm}, $G$ acts transitively on top dimensional simplices of $X$. As a result, for every $\lbrace i_0,...,i_k \rbrace \subseteq \I$, $G$ acts transitively on simplices of type $\lbrace i_0,...,i_k \rbrace$. Therefore  $\lbrace g_0 K_{\lbrace i_0 \rbrace},..., g_k  K_{\lbrace i_k \rbrace} \rbrace \in X(k)$ if and only if there is $g \in G$ such that for every $0 \leq j \leq k$, $g K_{\lbrace i_j \rbrace} = g_j K_{\lbrace i_0 \rbrace}$, i.e., 
$$\lbrace g_0 K_{\lbrace i_0 \rbrace},..., g_k  K_{\lbrace i_k \rbrace} \rbrace = \lbrace g K_{\lbrace i_0 \rbrace},..., g  K_{\lbrace i_k \rbrace} \rbrace.$$
As a result 
$$ X (k) = \bigcup_{\tau \subseteq \I, \vert \tau \vert = k+1} \lbrace \lbrace g K_{\lbrace i \rbrace} : i \in \tau \rbrace : g \in G \rbrace,$$
and $f$ is defined on every simplex in $X(k)$. To verify that $f$ is well defined, we note that if $\tau  \subseteq \I$ and $g_1, g_2 \in G$ such that
$$\lbrace g_1 K_{\lbrace i \rbrace} : i \in \tau \rbrace = \lbrace g_2 K_{\lbrace i \rbrace} : i \in \tau \rbrace.$$
Then for every $i \in \tau$, $g_1 K_{\lbrace i \rbrace} = g_2 K_{\lbrace i \rbrace}$, which is equivalent to $g_2^{-1} g_1 \in K_{\lbrace i \rbrace}$. Therefore
$$g_2^{-1} g_1 \in \bigcap_{\lbrace i \rbrace \subseteq \tau} K_{\lbrace i \rbrace} = K_\tau,$$
and $g_1 K_\tau = g_2 K_\tau$ as needed.

By a similar argument, $f$ is injective: if
$$f(\lbrace g_1 K_{\lbrace i \rbrace} : i \in \tau \rbrace) = f(\lbrace g_2 K_{\lbrace i \rbrace} : i \in \tau \rbrace),$$
then $g_1 K_\tau = g_2 K_\tau$. This implies that
$$g_2^{-1} g_1 \in  K_\tau = \bigcap_{\lbrace i \rbrace \subseteq \tau} K_{\lbrace i \rbrace},$$
and therefore
$$\lbrace g_1 K_{\lbrace i \rbrace} : i \in \tau \rbrace = \lbrace g_2 K_{\lbrace i \rbrace} : i \in \tau \rbrace,$$
as needed.

It is obvious that $f$ is onto and equivariant.

Let $\tau, \tau' \subseteq \I$ and $g, g' \in G$ such that $f^{-1} (g K_\tau) \subseteq f^{-1} (g' K_{\tau '})$. By the definition of $f$, this yields that
$$\lbrace g K_{\lbrace i \rbrace} : i \in \tau \rbrace \subseteq \lbrace g' K_{\lbrace i \rbrace} : i \in \tau' \rbrace,$$
which in turn yields that $\tau \subseteq \tau'$ and for every $i \in \tau$ the following holds: $g K_{\lbrace i \rbrace} = g' K_{\lbrace i \rbrace}$. Therefore
\begin{dmath}
\label{equiv-eq}
{\forall i \in \tau, g K_{\lbrace i \rbrace} = g' K_{\lbrace i \rbrace} \Leftrightarrow \forall i \in \tau, (g')^{-1} g K_{\lbrace i \rbrace} = K_{\lbrace i \rbrace} \Leftrightarrow} \\
{\forall i \in \tau, (g')^{-1} g \in K_{\lbrace i \rbrace}
\Leftrightarrow (g')^{-1} g \in \bigcap_{i \in \tau} K_{\lbrace i \rbrace} \Leftrightarrow (g')^{-1} g \in K_\tau.}
\end{dmath}

Conversely, assume that $\tau \subseteq \tau'$ and $(g')^{-1} g \in K_\tau$, then, by \eqref{equiv-eq}, for every $i \in \tau$, $g K_{\lbrace i \rbrace} = g' K_{\lbrace i \rbrace}$ and therefore
$$ \lbrace g K_{\lbrace i \rbrace} : i \in \tau \rbrace = \lbrace g' K_{\lbrace i \rbrace} : i \in \tau \rbrace  \subseteq \lbrace g' K_{\lbrace i \rbrace} : i \in \tau' \rbrace,$$
as needed.
\end{proof}

\begin{remark}
\label{different descrip remark}
The above proposition gives an alternative description of $X=X(G, (K_{\lbrace i \rbrace})_{i \in \I})$: for simplices $\eta, \sigma$ in $X$, denote $\eta \leq \sigma$ if $\eta$ is a face of $\sigma$ (we are not using the notation ``$\subseteq$'' to avoid confusion). Then, for every $0 \leq k \leq n$ we use the identification of function $f$ in the above proposition and define:
$$X (k) = \lbrace  g K_\tau : g \in G, \tau \in \I (k) \rbrace,$$
where $\I (k)$ defined as in Proposition \ref{different descrip prop}, with the relations: $g K_\tau \leq g' K_{\tau'}$ if and only if $\tau \subseteq \tau'$ and $(g')^{-1} g \in K_\tau$. 
\end{remark}

If $(G, (K_{\lbrace i \rbrace})_{i \in \I})$ is a subgroup geometry system, note that for every $\tau \subsetneqq \I$, $(K_\tau, (K_{\tau \cup \lbrace i \rbrace})_{i \in \I \setminus \tau})$ is also a subgroup geometry system. Therefore Theorem \ref{cost geom thm}, can be applied to yield a simplicial complex $X(K_{\tau}, (K_{\tau \cup \lbrace i \rbrace})_{i \in \I \setminus \tau})$. The next proposition states that the links of $X(G, (K_{\lbrace i \rbrace})_{i \in \I})$ can be described as those complexes.  
%The next proposition stated that for every simplex $\sigma$ of $X(G, (K_{\lbrace i \rbrace})_{i \in \I})$ of type $\tau \subsetneqq \I$, the link of $\sigma$ is isomorphic to $X(K_{\tau}, (K_{\tau \cup \lbrace i \rbrace})_{i \in \I \setminus \tau})$:
\begin{proposition}
\label{link coset prop}
Let $n \in \mathbb{N}$, $\I$ a set of cardinality $n+1$ and $G$ be group with subgroups $K_{\lbrace i \rbrace}, i \in \I$ such that $(G, (K_{\lbrace i \rbrace})_{i \in \I})$ is a subgroup geometry system. Given $g \in G$, $\tau \subsetneqq \I$, the link of $g K_{\tau}$ in $X(G, (K_{\lbrace i \rbrace})_{i \in \I})$ is isomorphic to $X(K_{\tau}, (K_{\tau \cup \lbrace i \rbrace})_{i \in \I \setminus \tau})$.
\end{proposition}

\begin{proof}
See \cite[Theorem 1.8.10 (ii)]{DiaGeomBook}.
\end{proof}

\subsection{The automorphism group of coset complexes}

Let $n \in \mathbb{N}$, $\I$ a finite set of cardinality $n+1$ (below, we will take without loss of generality $\I = \lbrace 0,..., n\rbrace$) and $G$ be group with subgroups $K_{\lbrace i \rbrace}, i \in \I$ such that $(G, (K_{\lbrace i \rbrace})_{i \in \I})$ is a subgroup geometry system. Denote $X = X(G, (K_{\lbrace i \rbrace})_{i \in \I})$. We have seen above that $G < \Aut (X)$. Below, we will show that given additional assumptions on on the automorphism group of $G$, $\Aut (X) \neq G$, i.e., $X$ has additional symmetries that do not arise from the action of $G$ on $X$. 

\begin{definition}
Let $G$ be a group with subgroups $H_1,...,H_l$. For $\gamma \in \Aut (G)$, we will say that $\gamma$ \emph{preserves the subgroups} $H_1,...,H_l$ if
$$\lbrace \gamma (H_i) : i=1,..., l \rbrace = \lbrace H_i : i=1,..., l \rbrace,$$
where $\gamma (H_i) = \lbrace \gamma (g) : g \in H_i \rbrace$.

Let $(G, (K_{\lbrace i \rbrace})_{i \in \I})$ be a subgroup geometry system. For $\gamma \in \Aut (G)$, we will say that \emph{preserves the subgroup geometry system}, if $\gamma$ preserves the subgroups $K_{\lbrace i \rbrace}, i \in \I$. For a group $\Gamma < \Aut (G)$, we will say that $\Gamma$ \emph{preserves the subgroup geometry system}, if each $\gamma \in \Gamma$ preserves the subgroup geometry system.
\end{definition}

\begin{proposition}
\label{gamma as a permutation prop}
Let $\I = \lbrace 0,...,n \rbrace$, $n \in \mathbb{N}$, $(G, (K_{\lbrace i \rbrace})_{i \in \I})$ a subgroup geometry system and $\Gamma < \Aut (G)$ that preserves the subgroup geometry system. Let
$\Sym (\lbrace 0,1,...,n \rbrace)$ be the permutation group of $\lbrace 0,...,n \rbrace$ and for every $\gamma \in \Gamma$, define $\psi_\gamma \in \Sym (\lbrace 0,1,...,n \rbrace)$ by
$$\gamma. K_{\lbrace i \rbrace} = K_{\lbrace \psi_\gamma (i) \rbrace}, \forall 0 \leq i \leq n.$$
Then the following holds:
\begin{enumerate}
\item $\psi_\gamma$ is well defined.
\item The map $\Psi : \Gamma \rightarrow \Sym (\lbrace 0,1,...,n \rbrace)$, $\Psi (\gamma) = \psi_{\gamma}$ is a group homomorphism.
\item For every $\tau \subseteq \I$, $\gamma . K_{\tau} = K_{\psi_\gamma (\tau)}$.
\end{enumerate}
\end{proposition}

\begin{proof}
First, we note that by $(\mathcal{A} 3)$ for every $i, j \in \I$, if $i \neq j$, then $K_{\lbrace i \rbrace} \neq K_{\lbrace j \rbrace}$. Therefore if $\gamma \in \Gamma$ preserves $K_{\lbrace 0 \rbrace},..., K_{\lbrace n \rbrace}$, then $\psi_\gamma$ defined above is indeed a permutation on the set $\lbrace 0,...,n\rbrace$.

Second, for $\gamma, \gamma' \in \Gamma$ and every $0 \leq i \leq n$,
\begin{dmath*}
K_{\lbrace \psi_{\gamma \gamma'} (i) \rbrace} =
(\gamma \gamma'). K_{\lbrace i \rbrace} =
\gamma. (\gamma'. K_{\lbrace i \rbrace}) =
\gamma. K_{\lbrace \psi_{\gamma'} (i) \rbrace} =
K_{\lbrace \psi_{\gamma} \psi_{\gamma'} (i) \rbrace},
\end{dmath*}
and therefore $\psi_{\gamma \gamma'} = \psi_{\gamma} \psi_{\gamma'}$ as needed.

Last, for every $\tau= \lbrace i_0,...,i_k \rbrace$,
$$K_{\lbrace i_0,...,i_k \rbrace} = K_{ \lbrace i_0 \rbrace} \cap K_{ \lbrace i_1 \rbrace} \cap ... \cap K_{ \lbrace i_k \rbrace}.$$
Therefore
\begin{dmath*}
\gamma .K_{\lbrace i_0,...,i_k \rbrace} = \gamma. K_{ \lbrace i_0 \rbrace} \cap  ... \cap \gamma. K_{ \lbrace i_k \rbrace} =
 K_{ \lbrace \psi_\gamma (i_0) \rbrace} \cap ... \cap K_{ \lbrace \psi_\gamma (i_k) \rbrace} = K_{\lbrace \psi_\gamma (i_0),...,\psi_\gamma (i_k) \rbrace},
\end{dmath*}
as needed.
\end{proof}

Recall that we saw above that we can define the simplices of $X$ as cosets (see Remark \ref{different descrip remark}). We will use this identification to show that given a group $\Gamma < \Aut (G)$ that preserves the subgroup geometry system, we have that $G \rtimes \Gamma < \Aut (X)$:

\begin{proposition}
\label{Aut prop}
Let $\I = \lbrace 0,...,n \rbrace$, $n \in \mathbb{N}$, $(G, (K_{\lbrace i \rbrace})_{i \in \I})$ a subgroup geometry system and $\Gamma < \Aut (G)$ that preserves the subgroup geometry system. For $X= X(G, (K_{\lbrace i \rbrace})_{i \in \I})$, $G \rtimes \Gamma$ acts simplicially on $X$ as follows: for every $0 \leq k \leq n$, every $\tau \in \I (k) = \lbrace \tau \subseteq \I : \vert \tau \vert = k+1 \rbrace$ and every $g' \in G$, define
 $(g, \gamma) . g' K_\tau = g \gamma (g') K_{\psi_\gamma (\tau)}$.
\end{proposition}

\begin{proof}
First, we will show that the action defined above is indeed a group action, i.e., that for every $g_1, g_2 \in G, \gamma_1, \gamma_2 \in \Gamma$, every $0 \leq k \leq n$, every $\tau \in \I (k)$ and every $g' \in G$, the following holds:
$$(g_1, \gamma_1). ((g_2,\gamma_2)) . g' K_\tau) =((g_1, \gamma_1) (g_2,\gamma_2)) . g' K_\tau .$$
Indeed, fix some $g_1, g_2, \gamma_1, \gamma_2, \tau, g'$ as above, then
\begin{dmath*}
(g_1, \gamma_1). ((g_2,\gamma_2)) . g' K_\tau) = (g_1, \gamma_1). (g_2 \gamma_2 (g') K_{\psi_{\gamma_2} (\tau)}) = g_1 \gamma_1 (g_2) \gamma_1 (\gamma_2 (g')) K_{\psi_{\gamma_1} (\psi_{\gamma_2} (\tau))} = (g_1 \gamma_1 (g_2)) (\gamma_1 \gamma_2) (g') K_{\psi_{\gamma_1 \gamma_2} (\tau)} =
(g_1 \gamma_1 (g_2), \gamma_1 \gamma_2). g' K_\tau =
((g_1, \gamma_1) (g_2,\gamma_2)) . g' K_\tau,
\end{dmath*}
where the last equality is due to the definition of semidirect product.

Second, we need to show that the action defined above preserves face relations, i.e., for every $\tau_1, \tau_2 \subseteq \I$ and every $g_1, g_2 \in G$, if $g_1 K_{\tau_1} \leq g_2 K_{\tau_2}$, then for every $g \in G, \gamma \in \Gamma$, $(g,\gamma). g_1 K_{\tau_1} \leq (g,\gamma). g_2 K_{\tau_2}$. Let $\tau_1, \tau_2, g_1, g_2$ as above such that $g_1 K_{\tau_1} \leq g_2 K_{\tau_2}$, i.e., such that $\tau_1 \subseteq \tau_2$ and $(g_2)^{-1} g_1 \in K_{\tau_1}$. We need to show that
$$\psi_\gamma (\tau_1) \subseteq \psi_\gamma (\tau_2) \text{ and } (g \gamma(g_2))^{-1} (g \gamma (g_1)) \in K_{\psi_\gamma (\tau_1)}.$$
Since $\psi_\gamma$ is a permutation, it is obvious that $\tau_1 \subseteq \tau_2$ implies $\psi_\gamma (\tau_1) \subseteq \psi_\gamma (\tau_2)$. Also, we note that
$$(g \gamma(g_2))^{-1} (g \gamma (g_1)) = \gamma(g_2^{-1} g_1) \in \gamma (K_{\tau_1}) = K_{\psi_\gamma (\tau_1)},$$
as needed.
\end{proof}

\subsection{Passing to quotients}
\label{Passing to quotients subsect}
Let $(G, (K_{\lbrace i \rbrace})_{i \in \I})$ be a subgroup geometry system and let $N \triangleleft G$ be a normal subgroup. The following proposition gives sufficient conditions in order to define a simplicial complex modeled over $G/N$ which has the same links as the complex $X(G, (K_{\lbrace i \rbrace})_{i \in \I})$. This can be seen as an analogue of passing to a quotient in the Cayley graph.

\begin{proposition}
\label{passing to quotient prop}
Let $(G, (K_{\lbrace i \rbrace})_{i \in \I})$ be a subgroup geometry system and let $N \triangleleft G$ be a normal subgroup. Assume that for every $i \in \I$, $K_{\lbrace i \rbrace} \cap N = \lbrace e \rbrace$. Then $(G/N, (K_{\lbrace i \rbrace} N / N)_{i \in \I})$ is a subgroup geometry system. Furthermore, if we denote $X_G = X(G, (K_{\lbrace i \rbrace})_{i \in \I})$ and $X_{G/N} = X(G/N, (K_{\lbrace i \rbrace} N / N)_{i \in \I})$, then the following hold:
\begin{enumerate}
\item For every $\tau \subseteq \I, \tau \neq \emptyset$ and every $g \in G$, the link of $(g N)(K_\tau N) / N$ in $X_{G/N}$ is isomorphic to the link of $g K_\tau$ in $X_G$, i.e., $X_G$ and $X_{G/N}$ have the same links (apart from the link of the empty simplex which is different).
\item $X_G$ is a covering of $X_{G/N}$.
\item If $G/N$ is a finite group, then $X_{G/N}$ is a finite complex and $\vert X_{G/N} (n) \vert \leq \vert G/N \vert$.
\end{enumerate}
\end{proposition}

\begin{proof}
By definition $K_\emptyset = G$ and therefore $K_\emptyset N / N = G N / N = G/N$. Let $\Phi : G \rightarrow G/ N$ denote the quotient map $\Phi (g) = g N$. We note that  the assumption that for every $\tau \subseteq \I, \tau \neq \emptyset$, $K_\tau \cap N = \lbrace e \rbrace$, implies that $\Phi$ is injective on every $K_\tau, \tau \subseteq \I, \tau \neq \emptyset$ and this implies that $(\mathcal{A} 1)-(\mathcal{A} 3)$ are fulfilled. 

Also by the injectivity of the quotient map $\Phi$ on every $K_\tau, \tau \subseteq \I, \tau \neq \emptyset$ and by Proposition \ref{link coset prop}, we deduce that the link of $K_\tau N / N$ is $X_{G/N}$ is isomorphic to the link of $K_{\tau}$ in $X_G$.

By the same reasoning, $\Phi$ induces a covering map of simplicial complexes: every simplex in $X_G$ labeled $g K_{\tau}$ is mapped to $(g N)(K_\tau N) / N$ in $X_{G/N}$.   

The last assertion follows from the fact that $G/N$ acts transitively on $X_{G/N} (n)$.
\end{proof}

\begin{observation}
Let $(G, (K_{\lbrace i \rbrace})_{i \in \I})$ be a subgroup geometry system and $N \triangleleft G$ be a normal subgroup such that for every $i \in \I$, $K_{\lbrace i \rbrace} \cap N = \lbrace e \rbrace$. We note that if $\Gamma < \Aut (G)$ preserves the subgroups $K_{\lbrace i \rbrace}, i \in \I$, and for every $\gamma \in \Gamma$, $\gamma (N) = N$, then (by abusing the notation) $\Gamma < \Aut (G /N)$ and it preserves the subgroups  $(K_{\lbrace i \rbrace} N) / N,  i \in \I$ and thus $(G/N) \rtimes \Gamma$ acts on $X_{G/N}$ as in Proposition \ref{Aut prop} above.
\end{observation}

\section{Elementary matrix groups with subgroup geometry systems}
\label{Elementary matrices construction sect}
Our main sources of new examples of groups with subgroup geometry systems are elementary matrix groups and Steinberg groups. We will show that such groups have subgroup geometry systems over any finitely generated algebra over a ring (and any choice of a generating set). We note that we could not find these systems in the literature regarding diagram geometries and, as far as we can tell, these are new examples of subgroup geometries. 

\subsection{subgroup geometry systems for elementary matrices}

Let $R$ be a unital commutative ring and $\mathcal{R}$ be a finitely generated $R$-algebra. Let $\lbrace t_1,...,t_l\rbrace$ be a generating set of $\mathcal{R}$ and let $T$ be the $R$-module generated by $\lbrace 1, t_1,...,t_l\rbrace$, i.e., $T= \lbrace r_0+ \sum_i r_i t_i : r_i \in R \rbrace$. Examples:
\begin{enumerate}
\item $\mathcal{R} = R$ with the generating set $\lbrace 1 \rbrace$ and $T = R$.
\item $\mathcal{R} = R[t]$ with the generating set $\lbrace 1,t \rbrace$ and $T = \lbrace a + bt : a,b \in R \rbrace$.
\item $\mathcal{R} = R\langle t_1,...,t_l \rangle$ (the ring of non-commutative polynomials) with the generating set $\lbrace 1,t_1,...,t_l \rbrace$ and $T = \lbrace a_0 + a_1 t_1 + ... + a_l t_l : a_0,...,a_l \in R \rbrace$.
\end{enumerate}

Let $n \in \mathbb{N}, n \geq 1$. Below we will consider $(n+1) \times (n+1)$ matrices with entries in $\mathcal{R}$ and for the sake of convenience we will index the rows/columns of those matrices with $0,...,n$ (and not with $1,...,n+1$).  For $0 \leq i, j \leq n, i \neq j$ and $r \in \mathcal{R}$, let $e_{i,j} (r)$ be the $(n+1) \times (n+1)$ matrix with $1$'s along the main diagonal, $r$ in the $(i,j)$ entry and $0$'s in all the other entries. The \textit{group of elementary matrices} denoted $\EL_{n+1} (\mathcal{R})$ is the group generated by the elementary matrices with coefficients in $\mathcal{R}$, i.e., $\EL_{n+1} (\mathcal{R}) = \langle e_{i,j} (r) : 0 \leq i, j \leq n, i \neq j, r \in \mathcal{R} \rangle$.

Let $R, \mathcal{R}, \lbrace 1, t_1,...,t_l\rbrace$ and $T$ as above, $n \geq 2$ be an integer and $\I = \lbrace 0,...,n \rbrace$. For $i \in  \I$ define $K_{\lbrace i \rbrace} < \EL_{n+1} (\mathcal{R})$ by
$$K_{\lbrace i \rbrace} = \langle e_{j, j+1} (m) : j \in \I \setminus \lbrace i \rbrace, m  \in T \rangle,$$
where $j+1$ is taken modulo $n+1$, i.e., for $j=n$, $e_{n,n+1} (m) = e_{n,0} (m)$. Below, we will show that  $(\EL_{n+1} (\mathcal{R}), ( K_{\lbrace i \rbrace})_{i \in \I})$ is a subgroup geometry system. In order to do so, we give an explicit description of the subgroups $K_\tau$ for $\tau \subseteq \I$.

\begin{lemma}
\label{K (n) lemma}
Under the above definitions, the group $K_{\lbrace n \rbrace}$ is the group of $(n+1) \times (n+1)$ upper triangular matrices $A = (a_{k,j}) \in \EL_{n+1} (\mathcal{R})$ with $1$'s along the main diagonal such that for every $0 \leq k<j \leq n$, $a_{k,j} \in T^{j-k}$ and $ T^{j-k}$ denotes the $R$-module of all the polynomials in $\lbrace 1, t_1,...,t_l\rbrace$ with degree $\leq j-k$.
\end{lemma}

\begin{proof}
Let $H < \EL_{n+1} (\mathcal{R})$ be the subgroup of $(n+1) \times (n+1)$ upper triangular matrices $A = (a_{k,j}) \in \EL_{n+1} (\mathcal{R})$ with $1$'s along the main diagonal such that for every $k<j$, $a_{k,j} \in T^{j-k}$ (we leave it to the reader to verify that this is in fact a subgroup, i.e., that it is closed under multiplication and inverse). We note that $K_{\lbrace n \rbrace}$ is generated by elements in $H$ and therefore $K_{\lbrace n \rbrace} \subseteq H$. For the reverse inclusion, we recall that multiplication of elementary matrices follow the (Steinberg) relation: for every pairwise  different $0 \leq j_1, j_2, j_3 \leq n$ and every $m_1, m_2 \in R$, 
$$[e_{j_1,j_2} (m_1), e_{j_2,j_3} (m_2)] = e_{j_1,j_3} (m_1 m_2).$$
Thus, for every $0 \leq k < j \leq n$ using this relation iteratively, yields that for every $m \in T^{j-k}$, $e_{k,j} (m) \in  K_{\lbrace n \rbrace}$. By row reduction, every matrix of $H$ can be written as a product of matrices of the form $\lbrace e_{k,j} (m): 0 \leq k < j \leq n, m \in T^{j-k} \rbrace$ and therefore $K_{\lbrace n \rbrace} = H$.
\end{proof}

The above lemma characterizes $K_{\lbrace n \rbrace}$, but a similar argument can be applied for every $i \in \I$ as stated in the following Lemma. The other cases are a little harder to describe (since they not upper triangular matrices) and leave the verification to the reader (this Lemma can also be deduced from the case $K_{\lbrace n \rbrace}$ using the automorphisms of $\EL_n (\mathcal{R})$ that preserve the $K_{\lbrace i \rbrace}$'s - see below). 

\begin{lemma}
For every $i \in \I$, $K_{\lbrace i \rbrace}$ is the group composed of all the matrices $A=(a_{k,j})$ such that
$$a_{k,j} \in \begin{cases}
\lbrace 1 \rbrace & k =j \\
T^{j-k} & k \neq j, i \notin \lbrace k, k+1,...,j-1 \rbrace \\
\lbrace 0 \rbrace & \text{otherwise}
\end{cases},$$
where $j-k$ and the elements of $\lbrace k, k+1,...,j-1 \rbrace$ are taken $\mod (n+1)$.
\end{lemma}

As a corollary, we have a description of all the groups of the form $K_\tau$:
\begin{corollary}
\label{subgroups of EL coro}
Let $0 \leq p <n$ and $0 \leq i_0 < i_1 <...<i_p \leq n$. For $\tau = \I \setminus \lbrace i_0, i_1,...,i_p \rbrace$, $K_\tau$ is the group composed of all the matrices $A=(a_{k,j})$ such that
$$a_{k,j} \in \begin{cases}
\lbrace 1 \rbrace & k =j \\
T^{j-k} & k \neq j, \lbrace k, k+1,...,j-1 \rbrace \subseteq \lbrace i_0, i_1,...,i_p \rbrace \\
\lbrace 0 \rbrace & \text{otherwise}
\end{cases},$$
where $j-k-1$, $k+1,...,j-1$ above are taken $\mod (n+1)$.
\end{corollary}

\begin{corollary}
\label{subgrp of EL generation coro}
For every $\tau \subseteq \I$, $K_{\tau} = \langle e_{j,j+1} (m) : j \in \I \setminus \tau, m  \in T \rangle \rangle$.
\end{corollary}

\begin{proof}
For $\tau = \emptyset$, it is a standard exercise that 
$$\EL_{n+1} (\mathcal{R}) = \langle e_{0,1} (m),...,e_{n-1,n} (m), e_{n,0} (m) : m \in T \rangle.$$ 
For $\tau \neq \emptyset$, $K_{\tau} = \langle e_{j,j+1} (m) : j \in \I \setminus \tau, m  \in T \rangle$ follows from Corollary \ref{subgroups of EL coro} by arguments similar to those of the proof of Lemma \ref{K (n) lemma}.
\end{proof}

After this corollary we are ready to prove the first part of Theorem \ref{geometry system for elementary mat groups thm - intro} stated in the introduction:

\begin{theorem}
\label{geometry system for elementary mat groups thm}
With the definitions above, $(\EL_{n+1} (\mathcal{R}), ( K_{\lbrace i \rbrace} )_{i \in \I})$ is subgroup geometry system. 
\end{theorem}

\begin{proof}
In order to prove that $(\EL_{n+1} (\mathcal{R}), ( K_{\lbrace i \rbrace} )_{i \in \I})$ is subgroup geometry system we need to verify conditions $(\mathcal{A} 1)- (\mathcal{A} 3)$. 

Condition $(\mathcal{A} 1)$ follows immediately from Corollary \ref{subgrp of EL generation coro}. 

Without loss of generality, it is enough to prove condition $(\mathcal{A} 2)$ for $i=n$, i.e., to prove that for every $\tau \subseteq \I$, $n \notin \tau$, 
$$K_{\tau} K_{\lbrace n \rbrace} = \bigcap_{i \in \tau} K_{\lbrace i \rbrace} K_{\lbrace n \rbrace}.$$
We note that this is equivalent to showing that 
$$\lbrace g K_{\lbrace n \rbrace} : g \in K_{\tau}\rbrace = \bigcap_{i \in \tau} \lbrace g K_{\lbrace n \rbrace} : g \in K_{\lbrace i \rbrace} \rbrace.$$
For every $\eta \subseteq \I$, $n \notin \eta$, denote $K_{\eta}^-$ to be the subgroup of $K_{\eta}$ which consists of all the matrices in $K_{\eta}$ with $0$'s above the main diagonal. We will first show that 
$$\lbrace g K_{\lbrace n \rbrace} : g \in K_{\eta}\rbrace = \lbrace g K_{\lbrace n \rbrace} : g \in K_{\eta}^- \rbrace.$$ 
Indeed, for every matrix $g=(a_{k,j}) \in K_{\eta}$, define a matrix $g ' =(b_{k,j})$ as 
$$b_{k,j} = \begin{cases}
a_{k,j} & k \leq j \\
0 & k>j
\end{cases}.$$
Note that $g'$ is an upper triangular matrix and by Corollary \ref{subgroups of EL coro}, $g' \in K_{\eta} \cap K_{\lbrace n \rbrace}$ and $g (g')^{-1} \in K_{\eta}^-$. Thus, for every $g \in  K_{\eta}$, $g K_{\lbrace n \rbrace} = g (g')^{-1} K_{\lbrace n \rbrace}$, with $g (g')^{-1} \in K_{\eta}^-$ as needed.

Given some fixed $\tau \subseteq \I$, $n \notin \tau$, the equality
$$\lbrace g K_{\lbrace n \rbrace} : g \in K_{\eta}\rbrace = \lbrace g K_{\lbrace n \rbrace} : g \in K_{\eta}^- \rbrace$$
is true for $\eta = \tau$ and $\eta = \lbrace i \rbrace$ for every $i \in \tau$. Furthermore, Corollary \ref{subgroups of EL coro} gives an explicit description of $K_{\tau}^-$ and $K_{\lbrace i \rbrace}$ for every $i \in \tau$ and from this description it follows that 
$$\lbrace g K_{\lbrace n \rbrace} : g \in K_{\tau}^- \rbrace =  \bigcap_{i \in \tau} \lbrace g K_{\lbrace n \rbrace} : g \in K_{\lbrace i \rbrace}^- \rbrace,$$
as needed.
 
Last, note that by Corollary \ref{subgroups of EL coro}, $K_{\I}$ is the trivial subgroup and 
$$K_{\I \setminus  \lbrace i \rbrace} = \lbrace e_{i,i+1} (m) : m \in T \rbrace,$$
and therefore $K_{\I \setminus  \lbrace i \rbrace} \neq K_{\I}$ and condition $(\mathcal{A} 3)$ is fulfilled.
\end{proof}

\begin{corollary}
With the above notations, $X=X(\EL_{n+1} (\mathcal{R}), ( K_{\lbrace i \rbrace} )_{i \in \I})$ is a pure $n$-dimensional, $(n+1)$-partite, strongly gallery connected clique complex with no free faces. Moreover, if $R$ is finite, then $X$ is locally finite.
\end{corollary}

\begin{proof}
The first part of the corollary follows directly from Theorem \ref{cost geom thm}. For the second part, note that if $R$ is finite, then by Corollary \ref{subgroups of EL coro}, $K_\tau$ is finite for every $\tau \subseteq \I, \tau \neq \emptyset$ and this in turn implies by Proposition \ref{link coset prop} that for every simplex $g K_\tau$ in $X$, the links $g K_\tau$ is finite.
\end{proof}

\subsection{Symmetries of the constructions}
\label{Symmetries of the constructions sect}
Here we will show that the complex $X(\EL_{n+1} (\mathcal{R}), ( K_{\lbrace i \rbrace} )_{i \in \I})$ that arises from the subgroup geometry system above has a rich symmetry group. 

We fix the following notation:  $K_{\lbrace i \rbrace}, i \in \I$ are the subgroups defined above and $X = X(\EL_{n+1} (\mathcal{R}), ( K_{\lbrace i \rbrace} )_{i \in \I})$.

By the definition of $X$, $\EL_{n+1} (\mathcal{R})$ acts transitively on top-dimensional simplices of $X$ and this action preserves the type of the simplices. Below, we will show that there is also a group $\Gamma < \Aut (\EL_{n+1} (\mathcal{R}))$ that acts on $X$ and is not type-preserving and that the transitivity properties of the action of $\EL_{n+1} (\mathcal{R}) \rtimes \Gamma$ on $X$ are, in some sense, the best that can be expected.

We recall that we showed in Proposition \ref{gamma as a permutation prop} that if $\Gamma < \Aut (\EL_{n+1} (\mathcal{R}))$ and $\Gamma$ preserves the subgroup geometry system, then there is a homomorphism $\Psi : \Gamma \rightarrow \Sym (\lbrace 0,...,n \rbrace)$ defined by $\Psi (\gamma) = \psi_\gamma$, where $\psi_\gamma$ is defined in Proposition \ref{gamma as a permutation prop}. Also, by Proposition \ref{Aut prop}, $\EL_{n+1} (\mathcal{R}) \rtimes \Gamma$ acts simplicially on $X$ by
$$(g,\gamma). g' K_{\tau} = g \gamma (g') K_{\psi_\gamma (\tau)}.$$

Below, we will see that in our construction an automorphism $\gamma$ that preserves the subgroup geometry system cannot be any permutation in $\Sym (\lbrace 0,...,n \rbrace)$:
\begin{proposition}
\label{limitations of sym prop}
Let $\gamma \in \Aut (\EL_{n+1} (\mathcal{R}))$ such that $\gamma$ preserves the subgroup geometry system. Then for every $0 \leq i,j \leq n, i \neq j$, $ (i -j) \mod (n+1)  = \pm 1$ if and only if $ (\psi_\gamma (i) - \psi_\gamma (j)) \mod (n+1)  = \pm 1$.
\end{proposition}

\begin{proof}
Let $\gamma \in \Aut (\EL_{n+1} (\mathcal{R}))$ such that $\gamma$ preserves the subgroup geometry system. Then by Proposition \ref{gamma as a permutation prop}, for every  $0 \leq i,j \leq n$, $i \neq j$, $\gamma$ induces an isomorphism between $K_{\I \setminus \lbrace i,j \rbrace}$ and $\gamma. K_{\I \setminus \lbrace i,j \rbrace} = K_{\I \setminus \lbrace \psi_\gamma (i), \psi_\gamma (j) \rbrace}$. We will show that a necessary condition for $K_{\I \setminus \lbrace i,j \rbrace}$ and $K_{\I \setminus \lbrace \psi_\gamma (i),\psi_\gamma (j) \rbrace}$ to be isomorphic is that  $ (i -j) \mod (n+1)  = \pm 1$ if and only if $ (\psi_\gamma (i) - \psi_\gamma (j)) \mod (n+1)  = \pm 1$.

We observe the following: first, if $(i -j) \mod (n+1)  = \pm 1$, then Corollary \ref{subgrp of EL generation coro}, the group
$$K_{\I \setminus \lbrace i,j \rbrace} = \langle  K_{\I \setminus \lbrace i \rbrace}, K_{\I \setminus \lbrace j \rbrace} \rangle,$$
is isomorphic to the group of $3 \time 3$ matrices of the form $H = \langle e_{0,1} (m), e_{1,2} (m) : m \in T \rangle$, i.e., 
\begin{dmath*}
H={\left\lbrace \left( \begin{array}{ccc}
1 & b_{0,1} & b_{0,2} \\
0 & 1 & b_{1,2} \\
0 & 0 & 1
\end{array} \right) : b_{0,1}, b_{1,2} \in T, b_{0,2} \in T^2 \right\rbrace.}
\end{dmath*}
We note that this group is not Abelian: for instance, $[e_{0,1} (1), e_{1,2} (1)] = e_{0,2} (1)$ (by definition, $1 \in T$). 
Second, if $ (i -j) \mod (n+1) \neq \pm 1$, then by the Corollary \ref{subgroups of EL coro}, the groups $K_{\I \setminus \lbrace i \rbrace}$ and $K_{\I \setminus \lbrace j \rbrace}$ are commuting Abelian groups and
$$K_{\I \setminus \lbrace i,j \rbrace} = \langle  K_{\I \setminus \lbrace i \rbrace}, K_{\I \setminus \lbrace j \rbrace} \rangle \cong K_{\I \setminus \lbrace i \rbrace} \times K_{\I \setminus \lbrace j \rbrace}.$$
This an an Abelian group and therefore cannot be isomorphic to $H$.

Therefore $K_{\I \setminus \lbrace i,j \rbrace} \cong K_{\I \setminus \lbrace \psi_\gamma (i), \psi_\gamma (j) \rbrace}$ implies that $ (i -j) \mod (n+1)  = \pm 1$ if and only if $ (\psi_\gamma (i) - \psi_\gamma (j)) \mod (n+1)  = \pm 1$.
\end{proof}

The above proposition leads us to recall the definition of the dihedral group:
\begin{definition}[The Dihedral group]
\label{dihedral group def}
The Dihedral group $D_{n+1}$ is the symmetry group of the regular $(n+1)$-gon. This group consist of $2(n+1)$ elements: $(n+1)$ rotations and $(n+1)$ reflections and can be written explicitly as
$$D_{n+1} = \langle r, s \vert r^{n+1} = s^2 = 1, srs=r^{-1} \rangle.$$  Equivalently, $D_{n+1}$ is the subgroup of the permutation group $\Sym ( \lbrace 0,...,n\rbrace)$ that fulfill the conditions of the above Proposition, i.e., $D_{n+1}$ is the subgroup of $\Sym ( \lbrace 0,...,n\rbrace)$ that consists of all the elements $\rho \in \Sym ( \lbrace 0,...,n\rbrace)$ for which $ (i -j) \mod (n+1)  = \pm 1$ if and only if $ (\rho (i) - \rho (j)) \mod (n+1)  = \pm 1$. Under this description, we can take the generators of $D_{n+1}$ to be $r,s \in \Sym ( \lbrace 0,...,n\rbrace)$, where $r$ is defined by $r(i)=i+1 \mod (n+1)$ and $s$ is defined by $s(i)=n-1-i \mod (n+1)$.
\end{definition}

A corollary of the above Proposition is:
\begin{corollary}
Let $K_{\I \setminus \lbrace 0 \rbrace},...,K_{\I \setminus \lbrace n \rbrace}$ as above and let $\Gamma < \Aut (\EL_{n+1} (\mathcal{R}))$ that preserves $K_{\I \setminus \lbrace 0 \rbrace},...,K_{\I \setminus \lbrace n \rbrace}$. Then $\Psi (\Gamma) \subseteq D_{n+1}$, where $\Psi : \Gamma \rightarrow \Sym (\lbrace 0,...,n \rbrace)$ is the homomorphism defined in Proposition \ref{gamma as a permutation prop}.
\end{corollary}

Below, we will construct a group $\Gamma$ that preserves $K_{\I \setminus \lbrace 0 \rbrace},...,K_{\I \setminus \lbrace n \rbrace}$ such that $\Psi (\Gamma) = D_{n+1}$. In order to construct this group, we recall some basic facts regarding the automorphism group of $EL_{n+1} (\mathcal{R})$. These facts can be summarized as follows: $\Aut (EL_{n+1} (\mathcal{R}))$ contains a subgroup that arises from permutation of the indices and the inverse-transpose involution. Below, we will explain this statement.

\paragraph{Permutations.} Let $\rho \in \Sym (\lbrace 0,...,n \rbrace)$ be a permutation. We recall that $\rho$ defines the following automorphism $\gamma_\rho$ on the group $EL_{n+1} (\mathcal{R})$: for any matrix $A \in EL_{n+1} (\mathcal{R})$, $\gamma_\rho (A)$ is the matrix defined as $(\gamma_\rho (A)) (i,j) = A (\rho^{-1} (i), \rho^{-1} (j))$. In order to verify that this is in fact an automorphism, we note that $\gamma_\rho (A) = P_{\rho}^{-1} A P_{\rho}$, where $P_{\rho}$ is the permutation matrix defined as
$$P_\rho (i,j) = \begin{cases}
1 & \rho (i) =j\\
0 & \text{otherwise}
\end{cases}.$$
We note that by this definition the group of permutations $\Sym (\lbrace 0,...,n \rbrace)$ is a subgroup of $\Aut (EL_{n+1} (\mathcal{R}))$: for every $\rho_1, \rho_2 \in \Sym (\lbrace 0,...,n \rbrace)$ and every $A \in EL_{n+1} (\mathcal{R})$, $(\gamma_{\rho_2} \gamma_{\rho_1}) (A) = \gamma_{\rho_2} (\gamma_{\rho_1} (A))$.

\paragraph{Inverse-transpose.} For any matrix $A \in EL_{n+1} (\mathcal{R})$, define $\iota (A) = (A^{-1})^t$. One can verify that $\iota (A) \in EL_{n+1} (\mathcal{R})$ since both the action of inverse and the action of transpose send elementary matrices to elementary matrices. Also, $\iota$ is and automorphism, because for every two matrices $A,B \in EL_{n+1} (\mathcal{R})$,
$$\iota (AB) = ((AB)^{-1})^{t} = (B^{-1} A^{-1})^t = (A^{-1})^t (B^{-1})^t = \iota (A) \iota (B).$$
Last, we note that $\iota^2 = e$.

With the above automorphisms at hand, we will construct a subgroup $\Gamma$ of $\Aut (EL_{n+1} (\mathcal{R}))$ that preserves subgroup geometry system defined above. Denote:
\begin{enumerate}
\item $\rho_{cyc} \in \Sym ( \lbrace 0,...,n\rbrace)$ is the cyclic permutation defined as $\rho_{cyc} (i) = i+1 \mod (n+1)$.
\item $\rho_{rev} \in \Sym ( \lbrace 0,...,n\rbrace)$ is the permutation defined as $\rho (i) = n-i \mod (n+1)$.
\end{enumerate}

Next, we are ready to prove the second part of Theorem \ref{geometry system for elementary mat groups thm - intro} stated in the introduction:
\begin{theorem}
\label{symmetry group onto dihedral thm}
Let $K_{\lbrace i \rbrace}, i=0,...,n$ be as above. Denote $\gamma_{r}, \gamma_{s} \in \Aut (G)$ to be the automorphisms defined as $\gamma_{r} = \gamma_{\rho_{cyc}}, \gamma_{s} = \iota \gamma_{\rho_{rev}}$.
Let $\Gamma < \Aut (\EL_{n+1} (\mathcal{R}))$ be the subgroup generated by $\gamma_{r}, \gamma_{s}$. Then $\Gamma$ preserves the subgroup geometry system and $\Psi (\Gamma) = D_{n+1}$.
\end{theorem}

\begin{proof}
We will prove the theorem by showing that $\gamma_r, \gamma_s$ preserve subgroup  $K_{\lbrace i \rbrace}, i=0,...,n$ and that $\Psi (\gamma_r) = r, \Psi (\gamma_s) =s$, where $r,s \in \Sym (\lbrace 0,...,n \rbrace)$ are the generators of the dihedral group defined in Definition \ref{dihedral group def} above.

For $j=0,...,n$ and $m \in T$, $\gamma_{r}.e_{j,j+1} (m) = e_{j+1,j+2} (m)$ (all the indices are taken $\mod (n+1)$). Recall that for every $i =0,...,n$, 
$$K_{\lbrace i \rbrace} = \langle e_{j,j+1} (m) : j \in \I \setminus \lbrace i \rbrace, m \in T \rangle,$$
and therefore, for every $i$,
\begin{dmath*}
\gamma_{r}.K_{\lbrace i \rbrace} = {\langle e_{j+1,j+2} (m) : j \in \I \setminus \lbrace i \rbrace, m \in T \rangle =}\\
{\langle e_{j,j+1} (m) : j \in \I \setminus \lbrace i+1 \rbrace, m \in T \rangle = K_{\lbrace i+1 \rbrace}.}
\end{dmath*}
As a result, $\gamma_{r}$ preserves the subgroup geometry and $\Psi (\gamma_{r}) = r \in D_{n+1}$.

Also, $j=0,...,n$ and $m \in T$, 
\begin{dmath*}
\gamma_{s}.e_{j,j+1} (m) = (\iota \gamma_{\rho_{rev}}).(e_{j,j+1} (m)) = \iota.(e_{n-j,n-j-1} (m))= ((e_{n-j,n-j-1} (m))^{-1})^t =
(e_{n-j,n-j-1} (-m))^t = e_{n-j-1,n-j} (-m).
\end{dmath*}
Therefore, for every $i=0,...,n$,
\begin{dmath*}
\gamma_{s}.K_{\lbrace i \rbrace} = {\langle e_{n-j-1,n-j} (-m) : j \in \I \setminus \lbrace i \rbrace, m \in T \rangle =}\\
{\langle e_{j,j+1} (m) : j \in \I \setminus \lbrace n-1-i \rbrace, m \in T \rangle = K_{\lbrace  n-1-i \rbrace}.}
\end{dmath*}
As a result, $\gamma_{s}$ preserves the subgroup geometry and $\Psi (\gamma_{s}) = s \in D_{n+1}$.
\end{proof}

The above Theorem gives rise to the following transitivity property of $\Aut (X)$:
\begin{corollary}
\label{transitive action coro 1}
Let $X= X(\EL_{n+1} (\mathcal{R}), ( K_{\lbrace i \rbrace} )_{i \in \I})$. For every $0 \leq k \leq n$, every $\tau, \tau' \subseteq \I $ with $\vert \tau \vert = \vert \tau ' \vert = k+1$ and every $g, g' \in \EL_{n+1} (\mathcal{R})$, if there is $\rho \in D_{n+1}$ such that $\rho (\tau) = \tau'$, then there is a simplicial automorphism in $\Aut (X)$ that sends $g K_{\tau}$ to $g' K_{\tau '}$ (recall that $g K_{\tau}$, $g' K_{\tau '}$ are identified with $k$-dimensional simplices in $X$).
\end{corollary}

\begin{proof}
Let $\tau, \tau', g, g'$ as above and let $\rho \in D_{n+1}$ such that $\rho (\tau) = \tau'$. By Theorem \ref{symmetry group onto dihedral thm}, there is $\gamma \in \Aut (\EL_{n+1} (\mathcal{R}))$ that preserves the subgroup geometry system such that $\psi_\gamma = \rho$. Then by Proposition \ref{Aut prop}, $( g' (\gamma (g))^{-1}, \gamma) \in G \rtimes \Gamma$ is in $\Aut (X)$ and
$$(g' (\gamma (g))^{-1} , \gamma).g K_{\tau} = g' (\gamma (g))^{-1} \gamma (g) K_{\psi_\gamma (\tau)} = g' K_{\tau '},$$
as needed.
\end{proof}

An immediate corollary of the above corollary is:
\begin{corollary}
\label{transitive action coro 2}
Let $X= X(\EL_{n+1} (\mathcal{R}), ( K_{\lbrace i \rbrace} )_{i \in \I})$. Then $\Aut (X)$ acts transitively on:
\begin{enumerate}
\item Top-dimensional simplices of $X$, i.e., on $X (n)$.
\item Vertices of $X$, i.e., on $X (0)$.
\item Simplices of co-dimension $1$ of $X$, i.e., on $X (n-1)$.
\end{enumerate}
\end{corollary}

Also, in the case $n=2$, the following can be deduced:
\begin{corollary}
Let $X= X(\EL_{3} (\mathcal{R}), ( K_{\lbrace i \rbrace} )_{i \in \I})$. Then for every vertex $v$ in $X$, the stabilizer of $v$ in $\Aut (X)$ acts transitively on the edges containing $v$.
\end{corollary}

\begin{proof}
For $n =2$, we note that $D_3 =  \Sym ( \lbrace 0,...,2\rbrace)$, thus in this case Theorem \ref{symmetry group onto dihedral thm} implies that $\Gamma < \Aut (\EL_{3} (\mathcal{R})), \Gamma = \langle \gamma_{r}, \gamma_{s} \rangle$ preserves the subgroup geometry system and includes all possible permutations on the subgroups $K_{\lbrace 0 \rbrace}, K_{\lbrace 1\rbrace},K_{\lbrace 2 \rbrace}$.

More explicitly, by symmetry it is enough to prove the assertion stated above for the vertex labeled $K_{\lbrace 2 \rbrace}$. For this vertex, the subgroup $K_{\lbrace 2 \rbrace}$ fixes the vertex $K_{\lbrace 2 \rbrace}$ and acts transitively on the edges of a fixed type, i.e., it acts transitively on edges of type $\lbrace 0 \rbrace$ and on edges of type $\lbrace 1 \rbrace$. Also, $\gamma_s$ defined above fixes the vertex $K_{\lbrace 2 \rbrace}$ and maps $K_{\lbrace 0 \rbrace}$ to $K_{\lbrace 1 \rbrace}$ and vice-versa. Therefore the subgroup $K_{\lbrace 2 \rbrace} \rtimes \lbrace e, \gamma_s\rbrace < G \rtimes \Gamma$ fixes the vertex $K_{\lbrace 2 \rbrace}$ and acts transitively on the edges attached to this vertex.
\end{proof}

While the above discussion concentrated on non-oriented simplices, Theorem \ref{symmetry group onto dihedral thm} has similar consequences for oriented simplices. We will explain this statement, but leave the proofs to the reader. For $0 \leq k \leq n$, let $\vec{\I} (k)$ denote oriented $k$-dimensional simplices of $\I$, i.e.,
$$\vec{\I} (k) = \lbrace (i_0,...,i_k) : \lbrace i_0,...,i_k \rbrace \in \I (k) \rbrace.$$
Also, let $\vec{X} (k)$ denote oriented $k$-dimensional simplices, i.e.,
$$\vec{X} (k) = \lbrace (v_0,...,v_k) : \lbrace v_0,...,v_k \rbrace \in X (k) \rbrace.$$
As in the non-oriented case, when $X$ arises from a subgroup geometry system $(G, (K_{\lbrace i \rbrace})_{i \in \I})$, $\vec{X} (k)$ can be canonically identified with a set of cosets, explicitly,
$$\vec{X} (k) = \lbrace g K_{\vec{\tau}} : g \in G,  \vec{\tau} \in \vec{\I} (k) \rbrace.$$
A corollary of Theorem \ref{symmetry group onto dihedral thm} (which is just an adaptation of Corollary \ref{transitive action coro 1} the oriented setting) is:
\begin{corollary}
\label{transitive oriented action coro 1}
Let $X= X(\EL_{n+1} (\mathcal{R}), ( K_{\lbrace i \rbrace} )_{i \in \I})$. For every $0 \leq k \leq n$, every $\vec{\tau}, \vec{\tau} ' \in \vec{\I} (k)$ and every $g, g' \in \EL_{n+1} (\mathcal{R})$, if there is $\rho \in D_{n+1}$ such that $\rho (\vec{\tau}) = \vec{\tau} '$, then there is a simplicial automorphism in $\Aut (X)$ that sends $g K_{\vec{\tau}}$ to $g' K_{\vec{\tau} '}$.
\end{corollary}

\begin{remark}
We note that Corollary \ref{transitive action coro 1} does not imply that in general $\EL_{n+1} (\mathcal{R}) \rtimes \Gamma= \Aut (X)$, but only $\EL_{n+1} (\mathcal{R}) \rtimes \Gamma \leq \Aut (X)$. First, there might be automorphism of $X$ that do not arise from the action or $\EL_{n+1} (\mathcal{R})$ or of $\Aut (\EL_{n+1} (\mathcal{R}))$. Second, there can be non-trivial elements of $\Gamma$ that are mapped by $\Psi$ to the identity. For instance, for $\mathcal{R}= \mathbb{F}_q [t]$, where $q=p^m$ ($p$ is prime) and $m>1$, there are $m-1$ non-trivial automorphisms of $\mathbb{F}_q$. Each one of this non-trivial automorphisms, induces an automorphism of $\EL_{n+1} (\mathbb{F}_q [t])$ (by acting on the coefficients) that map each group $K_\tau$, $\tau \subseteq \I$ to itself, but acts non-trivially on $X$. %Below, we will note that working with $\mathbb{F}_{q} [t_1,...,t_m]$ instead of $\mathbb{F}_q [t]$ adds
\end{remark}

\begin{remark}
\label{symmetries comparison remark}
It is worth while comparing Corollary \ref{transitive action coro 1} (and Corollary \ref{transitive oriented action coro 1}) to the classical case of $\widetilde{A}_n$ Buildings. For a non-archimedean local field $F$, the group $PGL_{n+1} (F)$ acts on the $n$-dimensional $\widetilde{A}_n$ Building $X$. The transitivity properties of this action are similar to those discussed in Corollary \ref{transitive action coro 1}: every simplex $\tau \in X(k)$, has a type $\lbrace i_0,...,i_k \rbrace \subseteq \lbrace 0,...,n \rbrace$ and given $\tau, \tau ' \in X(k)$, with types $\lbrace i_0,...,i_k \rbrace, \lbrace i_0 ',...,i_k ' \rbrace$, there is $g \in PGL_{n+1} (F)$ such that $g. \tau = \tau '$, if and only if there is $\rho \in D_{n+1}$ such that $\lbrace \rho (i_0),..., \rho (i_k) \rbrace = \lbrace i_0 ',...,i_k ' \rbrace$. However, the case of a classical $\widetilde{A}_n$ Building, the action is not simply transitive on top dimensional simplices of the building. In other words, given an $n$-dimensional simplex, its stabilizer group is non-trivial, i.e., it contains elements other the the identity. In our case, we do not know if the automorphism group of $X(\EL_{n+1} (\mathcal{R}), ( K_{\lbrace i \rbrace} )_{i \in \I})$ always contains a non-trivial group element fixing a top dimensional simplex.
\end{remark}

\begin{comment}
In this case, there is a subgroup $K$ of $PGL_{n+1} (F)$ pointwise stabilizing an $n$-dimensional simplex $\I \in X(n)$ and $X$ is identified with $X \cong PGL_{n+1} (F) /K$ and $PGL_{n+1} (F)$ therefore acts on $X$ from the left.

However, when passing to a finite Ramanujan complex this symmetry is \underline{not} known to be preserved: passing to Ramanujan complex is done by taking a lattice $\Lambda <  PGL_{n+1} (F)$ and passing to the quotient \\
$\Lambda \backslash X = \Lambda \backslash PGL_{n+1} (F) /K$. We note that $PGL_{n+1} (F)$ can not act on \\
$\Lambda \backslash PGL_{n+1} (F) /K$ and therefore the transitivity properties of this action on $X$ are a priori lost when passing to the quotient. In \cite{LSV1}, the authors discuss choosing a lattice for which the quotient does admit some transitivity properties, but those do not have the full generality of Corollary \ref{transitive action coro 1} or Corollary \ref{transitive oriented action coro 1} in our construction.
\end{comment}

\subsection{Subgroup geometry systems for Steinberg groups}

Let $\mathcal{R}, R,T, \I$ be as above. The \textit{Steinberg group} denoted $\St_{n+1} (\mathcal{R})$, is the group generated by $x_{i,j} (m)$ where $0 \leq i, j \leq n$, and $m \in \mathcal{R}$ with the Steinberg relations:
\begin{enumerate}
\item For every $0 \leq i, j \leq n$ and every $m_1, m_2 \in \mathcal{R}$
$$x_{i,j} (m_1) x_{i,j} (m_2) = x_{i,j} (m_1+m_2).$$
\item For every $0 \leq i, j, k \leq n$ such that $i \neq k$ and every $m_1, m_2 \in \mathcal{R}$
$$[x_{i,j} (m_1), x_{j,k} (m_2)] = x_{i,k} (m_1 m_2).$$
\item For every $0 \leq i, j,k,s \leq n$ such that $i \neq k$ and $j \neq s$ and every $m_1, m_2 \in \mathcal{R}$
$$[x_{i,j} (m_1), x_{s,k} (m_2)] = 1.$$
\end{enumerate}

One can verify that $\EL_{n+1} (\mathcal{R})$ is a quotient of $\St_{n+1} (\mathcal{R})$ given by the map $x_{i,j} (m) \rightarrow e_{i,j} (m)$.  Define for $i \in \I$ the subgroups $K_{\lbrace i \rbrace} = \langle  x_{j,j+1} (m) :m \in T, j \in \I \setminus \lbrace i \rbrace \rangle$. We leave it to the reader to verify the following:
\begin{enumerate}
\item The map $x_{i,j} (m) \rightarrow e_{i,j} (m)$ restricted to $K_{\lbrace i \rbrace}$, $i \in \I$ is an isomorphism.
\item $(\St_{n+1} (\mathcal{R}), (K_{\lbrace i \rbrace})_{i \in \I})$ is a subgroup geometry system.
\end{enumerate}
Therefore by Proposition \ref{passing to quotient prop}, $X (\St_{n+1} (\mathcal{R}), (K_{\lbrace i \rbrace})_{i \in \I})$ is a covering of \\
$X(\EL_{n+1} (\mathcal{R}), (K_{\lbrace i \rbrace})_{i \in \I})$. Also, one can verify that the automorphism group of $X (\St_{n+1} (\mathcal{R}), (K_{\lbrace i \rbrace})_{i \in \I})$ contains $\St_{n+1} (\mathcal{R}) \rtimes D_{n+1}$ in the same manner this was done for $X(\EL_{n+1} (\mathcal{R}), (K_{\lbrace i \rbrace})_{i \in \I})$, i.e., by showing that $\gamma_{r} (x_{i,j} (m)) = x_{i+1,j+1} (m), \gamma_s (x_{i,j} (m)) = x_{n-j,n-i} (m)$ are automorphisms of $\St_{n+1} (\mathcal{R})$ that preserve the subgroup geometry system. In this case, we need to verify that $\gamma_{r}, \gamma_{s}$ preserve the Steinberg relations defined above and thus act as automorphisms of $\St_{n+1} (\mathcal{R})$ (this is different from the case of $\EL_{n+1} (\mathcal{R})$, where we presented $\gamma_r, \gamma_s$ as (products of) well-known automorphisms of $\EL_{n+1} (\mathcal{R})$). Such direct verification is straight-forward and thus left for the reader.

\section{Construction of local spectral high dimensional expanders}
\label{Construction of HD exp sec.}

Let $X$ a pure $n$-dimensional finite simplicial complex. For every $-1 \leq k \leq n-2$, $\tau \in X(k)$, the one skeleton of $X_{\tau}$ is a graph with a weight function on the edges corresponding to the high dimensional structure of $X$. Explicitly, given $\tau \in X(k)$ and $\lbrace u,v \rbrace \in X_\tau (1)$, we set
$$w(\lbrace u,v \rbrace) = (n-k-2)! \vert \lbrace \sigma \in X_{\tau} (n-1-k) : \lbrace u,v \rbrace \subseteq \sigma \rbrace \vert.$$
Let $\mu_{\tau}$ denote the second largest eigenvalue of the random walk on $X_\tau$ with respect to the weight on the edges.

\begin{definition}[one-sided local spectral expander]
For $0 \leq \lambda < 1$, a pure $n$-dimensional finite simplicial complex $X$ is a {\em one-sided $\lambda$-local-spectral expander} if for every $-1 \leq k \leq n-2$, and for every $\tau \in X(k)$,  $\mu_{\tau} \leq \lambda$.
\end{definition}

As in the case of expander graphs, our objective is not to construct a single one-sided $\lambda$-local-spectral expander, but given $0 < \lambda <1$ and $n>1$, to construct a family of pure $n$-dimensional finite simplicial complexes $\lbrace X^{(s)} \rbrace_{s \in \mathbb{A}}$ where $\mathbb{A} \subseteq \mathbb{N}$ is an infinite set such that the following hold:
\begin{enumerate}
\item For every $s \in \mathbb{A}$, $X^{(s)}$ is a one-sided $\lambda$-local-spectral expander.
\item The family $\lbrace X^{(s)} \rbrace_{s \in \mathbb{A}}$ has uniformly bounded degree in the following sense: 
$$\sup_{s} \max_{\lbrace v \rbrace \in X^{(s)} (0)} \vert \lbrace \sigma \in X^{(s)} (n) : v \in \sigma \rbrace \vert < \infty .$$
\item The number of vertices tends to $\infty$ with $s$: 
$$\lim_{s \in \mathbb{A}, s \rightarrow \infty} \vert X^{(s)} (0) \vert = \infty.$$
\end{enumerate}

\begin{remark}
Note that the condition of uniformly bounded degree can be rephrased as follows: there is a constant $M$ such that for every $s$, every $0 \leq k \leq n-2$ and every $\tau \in X^{(s)} (k)$, the $1$-skeleton of $X^{(s)}_\tau$ has a degree $\leq M$ (here degree is taken in the usual sense - the degree of the graph).
\end{remark}

In \cite{OppLocI}, the second named author proved that for every $\lambda>0$, and every simplicial complex $X$, if all the links of $X$ of dimension $\geq 1$ are connected and if for every $1$-dimensional link the second largest eigenvalue of the simple random walk is $\leq \frac{\lambda}{1+(n-1)\lambda}$, then $X$ is a one-sided $\lambda$-local spectral expander. In order to construct one-sided local spectral expander using subgroup geometry systems, we need only to verify the spectral condition on $1$-dimensional links (connectivity of $X$ and of every link of dimension $\geq 1$ follows from Theorem \ref{cost geom thm}). 

Recall that in our construction, the $1$-dimensional links arose from a subgroup geometry (sub)system (see Proposition \ref{link coset prop}). Therefore we are lead to the following question: given a subgroup geometry system $(G, ( K_{\lbrace 0 \rbrace}, K_{\lbrace 1 \rbrace} ))$, how can we bound the spectrum of $X (G, ( K_{\lbrace 0 \rbrace}, K_{\lbrace 1 \rbrace} ))$? (in this case, $X (G, ( K_{\lbrace 0 \rbrace}, K_{\lbrace 1 \rbrace} ))$ is a finite biregular, bipartite graph - see Proposition \ref{incidence in cosets prop} below). 

Since $X (G, ( K_{\lbrace 0 \rbrace}, K_{\lbrace 1 \rbrace} ))$ is determined by the subgroup geometry system, it is not surprising that the bound of the spectrum is achieved via a property of the subgroup geometry system. The simplest bound on the spectrum of $X (G, ( K_{\lbrace 0 \rbrace}, K_{\lbrace 1 \rbrace} ))$ is achieved when $K_{\lbrace 0 \rbrace}$ and $K_{\lbrace 1 \rbrace}$ commute:
\begin{proposition}
\label{commuting subgrps spec bound prop}
Let $G$ be a finite group with a subgroup geometry system $(G, ( K_{\lbrace 0 \rbrace}, K_{\lbrace 1 \rbrace}))$ such that $K_{\lbrace 0 \rbrace}$ and $K_{\lbrace 1 \rbrace}$ commute, i.e., for every $g_i \in K_{\lbrace i \rbrace}$, $i=0,1$, $g_0 g_1 = g_1 g_0$. Then $X (G, ( K_{\lbrace 0 \rbrace}, K_{\lbrace 1 \rbrace} ))$ is a complete bipartite graph and therefore the spectrum of the simple random walk on $X_G$ is $\lbrace -1, 0, 1\rbrace$ and in particular the second largest eigenvalue is $0$.
\end{proposition}

\begin{proof}
Here we will use the description of $X (G, ( K_{\lbrace 0 \rbrace}, K_{\lbrace 1 \rbrace} ))$ given in Remark \ref{different descrip remark}. 

Recall that by the $(\mathcal{A} 1)$, $G = \langle  K_{\lbrace 0 \rbrace}, K_{\lbrace 1 \rbrace} \rangle$. Furthermore, if $K_{\lbrace 0 \rbrace}$ and $K_{\lbrace 1 \rbrace}$ commute, then every $g \in G$ can be written as $g = g_0 g_1 = g_1 g_0$, where $g_0 \in K_{\lbrace 0 \rbrace}, g_1 \in K_{\lbrace 1 \rbrace}$. 

This implies that every vertex of type $0$ is of the form $g_1 K_{\lbrace 0 \rbrace}$ where $g_1 \in K_{\lbrace 1 \rbrace}$ and every vertex of type $1$ is of the form $g_0 K_{\lbrace 1 \rbrace}$ where $g_0 \in K_{\lbrace 0 \rbrace}$. 

Given two vertices of different types $g_1 K_{\lbrace 0 \rbrace}, g_0 K_{\lbrace 1 \rbrace}$, we claim that they are connected by the edge $g_1 g_0 K_{\lbrace 0,1 \rbrace}$. Recall that by Remark \ref{different descrip remark}, to show that  $g_1 K_{\lbrace 0 \rbrace}, g_0 K_{\lbrace 1 \rbrace}$ are connected by this edge, we need to verify that $(g_1 g_0)^{-1} g_0 \in K_{\lbrace 1 \rbrace}$ and $(g_1 g_0)^{-1} g_1 \in K_{\lbrace 0 \rbrace}$, but this is obvious in the second case and follows easily from commutativity in the first case. The fact that a complete bipartite graph has the spectrum $\lbrace -1, 0, 1\rbrace$ is standard (see for instance \cite[Example 1.2]{ChungBook}).
\end{proof}

In more complicated cases, in order to bound the spectrum of \\ $X (G, ( K_{\lbrace 0 \rbrace}, K_{\lbrace 1 \rbrace} ))$, we will need the notion of $\varepsilon$-orthogonality between subgroups. 
%Below, we will just give the basic definitions and results regarding $\varepsilon$-orthogonality needed for our construction of local spectral expanders. The proofs and the more technical discussion can be found in Appendix REF.

\begin{definition}
\label{varepsilon orthogonal subspaces def}
Let $\mathcal{H}$ be a Hilbert space and let $\mathcal{H}_1, \mathcal{H}_2 \subseteq \mathcal{H}$ be closed subspaces of $\mathcal{H}$. Denote
\begin{dmath*}
\theta (\mathcal{H}_1, \mathcal{H}_2) =
{\inf \lbrace \kappa :  \vert \langle x_1, x_2 \rangle \vert \leq \kappa \Vert x_1 \Vert \Vert x_2 \Vert}, \\ {\forall x_1 \in \mathcal{H}_1 \cap (\mathcal{H}_1 \cap  \mathcal{H}_2)^\perp, \forall x_2 \in \mathcal{H}_2 \cap (\mathcal{H}_1 \cap  \mathcal{H}_2)^\perp  \rbrace}.
\end{dmath*}
For $0 \leq \varepsilon \leq 1$, $\mathcal{H}_1, \mathcal{H}_2$ will be called $\varepsilon$-orthogonal if $\theta (\mathcal{H}_1, \mathcal{H}_2) \leq \varepsilon$.
\end{definition}

\begin{definition}
Let $G$ be a group with subgroups $K_{\lbrace 0 \rbrace}, K_{\lbrace 1 \rbrace} <G$, such that $\langle K_{\lbrace 0 \rbrace}, K_{\lbrace 1 \rbrace} \rangle = G$. For a unitary representation $(\pi, \mathcal{H})$, we denote
$$\mathcal{H}^{\pi (K_{\lbrace 0 \rbrace})} = \lbrace x \in \mathcal{H} : \forall g \in K_{\lbrace 0 \rbrace}, \pi (g) x = x \rbrace,$$
$$\mathcal{H}^{\pi (K_{\lbrace 1 \rbrace})} = \lbrace x \in \mathcal{H} : \forall g \in K_{\lbrace 1 \rbrace}, \pi (g) x = x \rbrace.$$

$K_{\lbrace 0 \rbrace}, K_{\lbrace 1 \rbrace}$ are called $\varepsilon$-orthogonal if for every unitary representation $(\pi, \mathcal{H})$ of $G$, $\mathcal{H}^{\pi (K_{\lbrace 0 \rbrace})}$ and $\mathcal{H}^{\pi (K_{\lbrace 1 \rbrace})}$ are $\varepsilon$-orthogonal.
\end{definition}

Using the definition of $\varepsilon$-orthogonality, we can bound the spectrum of $X(G, ( K_{\lbrace 0 \rbrace}, K_{\lbrace 1 \rbrace} ))$:
\begin{theorem}
\label{orthogonality-spec thm}
Let $G$ be a finite group with a subgroup geometry system $(G, ( K_{\lbrace 0 \rbrace}, K_{\lbrace 1 \rbrace} ))$. Also, let $\lambda$ to be the second largest eigenvalue of the simple random walk on $X(G, ( K_{\lbrace 0 \rbrace}, K_{\lbrace 1 \rbrace} ))$. For every $0 \leq \varepsilon$, $\lambda \leq \varepsilon$ if and only if $K_{ \lbrace 0 \rbrace}$ and $K_{ \lbrace 1 \rbrace}$ are $\varepsilon$-orthogonal.
\end{theorem}

This Theorem was known to some experts: a partial proof of this Theorem appears in \cite{DymaraJ}[proof of Lemma 4.6] and this Theorem can also be deduced from the work of the second named author in \cite{OppAve}[Section 4.3]. However, we could not find an explicit proof of this Theorem in the literature and therefore we provided a proof in Appendix \ref{spectral gap appen}. 

\begin{remark}
Even in cases where $X$ is infinite, but locally finite, obtaining a bound on the spectrum of the links (or equivalently on the constant of orthogonality between subgroups) is desirable, because it leads to various rigidity results (see for instance: \cite{DymaraJ}, \cite{ErshovJZ}, \cite{OppWeig}, \cite{OppAve}, \cite{OppVan}, \cite{ErshovRall}).
\end{remark}

In \cite{ErshovJZ}, Ershov and Jaikin-Zapirain developed a theory of $\varepsilon$-orthogonality between subgroups in the case where $G$ is of nilpotency class two and $K_{\lbrace 0 \rbrace}, K_{\lbrace 1 \rbrace}$ are Abelian. We will not state their result in the fullest generality, but only the version we will use below:
\begin{theorem}\cite[Corollary 4.7]{ErshovJZ}
Let $R$ be a Noetherian ring, $\mathcal{R}$ be a finitely generated $R$-algebra, $\lbrace 1,t_1,...,t_l \rbrace$ be a generating set of $\mathcal{R}$ and $T = \lbrace r_0 + \sum_i r_i t_i : r \in R \rbrace$. Also, let $G$ be the upper-triangular group 
$$G = \left\lbrace \left( \begin{array}{ccc} 
1 & a_{0,1} & a_{0,2} \\
0 & 1 & a_{1,2} \\
0 & 0 & 1
\end{array} \right) : a_{0,1}, a_{1,2} \in T, a_{0,2} \in T^2 \right\rbrace,$$
and $K_{\lbrace 0 \rbrace} = \lbrace e_{0,1} (a) : a \in T \rbrace, K_{\lbrace 1 \rbrace} = \lbrace e_{1,2} (a) : a \in T \rbrace$. If the smallest proper submodule of $T$ is of index $q$, then $K_{\lbrace 0 \rbrace}$ and $K_{\lbrace 1 \rbrace}$ are $\frac{1}{\sqrt{q}}$-orthogonal.
\end{theorem}

In particular, this yields the following corollary that we will use in our construction below:
\begin{corollary}
\label{spec bound coro}
Let $q$ be a prime power, $G$ be the upper-triangular group 
$$G = \left\lbrace \left( \begin{array}{ccc} 
1 & a_0 + a_1 t & c_0 + c_1 t + c_2 t^2 \\
0 & 1 & b_0 + b_1 t \\
0 & 0 & 1
\end{array} \right) : a_0,...,c_2 \in \mathbb{F}_q \right\rbrace,$$
and $K_{\lbrace 0 \rbrace} = \lbrace e_{0,1} (a_0 + a_1 t) :a_0, a_1 \in \mathbb{F}_q \rbrace, K_{\lbrace 1 \rbrace} = \lbrace e_{1,2} (b_0 + b_1 t) :b_0, b_1 \in \mathbb{F}_q \rbrace$. Then $K_{\lbrace 0 \rbrace}$ and $K_{\lbrace 1 \rbrace}$ are $\frac{1}{\sqrt{q}}$-orthogonal.
\end{corollary}

After this, we can use the construction of \cref{Elementary matrices construction sect} to produce a construction of a family of local spectral expanders. 
\begin{theorem}
Let $n \geq 2$ and let $q$ be a prime power such that $q > (n-1)^2$. For $s \in \mathbb{N}$, let $\mathbb{F}_q [t] / \langle t^s \rangle$ be the $\mathbb{F}_q$ algebra with the generating set $\lbrace 1,t \rbrace$. Let $X^{(s)}$ be the simplicial complex of the subgroup geometry system of $\EL_{n+1} ( \mathbb{F}_q [t] / \langle t^s \rangle)$ described is \cref{Elementary matrices construction sect} above. Then for every $s > n$, the following hold for $X^{(s)}$:
\begin{enumerate}
\item $X^{(s)}$ is a pure $n$-dimensional, $(n+1)$-partite, strongly gallery connected clique complex with no free faces.
\item $X^{(s)}$ is finite and the number of vertices of $X^{(s)}$ tends to infinity as $s$ tends to infinity.
\item There is a constant $Q=Q(q)$ such that for every $s$, each vertex of $X^{(s)}$ is contained in exactly $Q$ $n$-dimensional simplices.
\item Each $1$-dimensional link of $X^{(s)}$ has a spectral gap greater or equal  to $1-\frac{1}{\sqrt{q}}$.
\item $X^{(s)}$ is a $\frac{1}{\sqrt{q}-(n-1)}$-local spectral expander. In particular, the spectral gap of the one-skeleton of $X^{(s)}$ is greater or equal  to $1-\frac{1}{\sqrt{q}-(n-1)}$.
\item The automorphisms group of $X^{(s)}$ contains $\EL_{n+1} (\mathbb{F}_q [t] / \langle t^s \rangle) \rtimes D_{n+1}$, where $\EL_{n+1} (\mathbb{F}_q [t] / \langle t^s \rangle)$ acts transitively on $n$-dimensional simplices and $D_{n+1}$ rotates / inverts the types. 
\end{enumerate}
Thus, given any $\lambda >0$, we can take $q$ large enough such that $\lbrace X^{(s)} : s > 2^{n-1} \rbrace $ will be an infinite family of $n$-dimensional $\lambda$-local spectral expanders which are uniformly locally finite.
\end{theorem}

\begin{proof}
The first assertion is due to Theorem \ref{cost geom thm}. 

The second assertion follows for the fact that $ \mathbb{F}_q [t] / \langle t^s \rangle$ is finite and $\vert  \mathbb{F}_q [t] / \langle t^s \rangle \vert$ tends to infinity with $s$. 

For the third assertion, note that for every $s > n$, $T^{n} / \langle t^s \rangle = T^{n}$ where $T = \lbrace a_0 + a_1 t :a_0, a_1 \in \mathbb{F}_q \rbrace$. Thus by Corollary \ref{subgroups of EL coro}, the subgroups $K_{\lbrace i \rbrace}$, $i \in \I$ are the same for every $s > n$ and moreover, for every $i, i'$, $K_{\lbrace i \rbrace}, K_{\lbrace i' \rbrace}$ are isomorphic and therefore the links of any two vertices are isomorphic and does not change with $s$. In particular, for $Q= \vert K_{\lbrace i \rbrace} \vert$ each vertex of $X^{(s)}$ is contained in exactly $Q$ $n$-dimensional simplices.

For the forth assertion we note that every $1$-dimensional link is of the form $X (K_{\I \setminus \lbrace i,j \rbrace}, (K_{\I \setminus \lbrace i \rbrace}, \I \setminus \lbrace j \rbrace))$ where $0 < i < j \leq n$. As in the proof of Proposition \ref{limitations of sym prop}, there are two possibilities: either $\vert i-j \mod (n+1) \vert >1$ and then $K_{\I \setminus \lbrace i,j \rbrace}$ is Abelian and in that case by Proposition \ref{commuting subgrps spec bound prop}, the link is a full bipartite graph and the second largest eigenvalue is $0$. Or $\vert i-j \mod (n+1) \vert =1$ and in that case $K_{\I \setminus \lbrace i,j \rbrace}$ is the group of Corollary \ref{spec bound coro} and therefore the second eigenvalue is $\frac{1}{\sqrt{q}}$.

The fifth assertion follows directly from the main result of \cite{OppLocI} mentioned above.

The last assertion is due to Theorem \ref{symmetry group onto dihedral thm}.
\end{proof}

\bibliographystyle{alpha}
\bibliography{bibl}

\appendix

\section{Incidence systems}
\label{Incidence systems appen}
A parallel terminology to the terminology of simplicial complexes is the terminology of incidence systems. Since virtually all the literature regarding coset geometry is written in the language of incidence systems, we use this appendix to provide a dictionary between the two terminologies.

A triplet $(V,*,\type)$ is called an \textit{incidence system} over a finite set $\I$ if:
\begin{enumerate}
\item $V$ is a set.
\item $*$ is a reflexive and symmetric relation called the \textit{incidence relation}. If the pair $(u,v)$ is in the incidence relation of the system, we write $u*v$ and say that $u,v$ are incident.
\item $\type : V \rightarrow \I$ is a function such that for $v,u \in V$, $u*v$ implies $\type (u) \neq \type (v)$.
\end{enumerate}

Let $(V,*,\type)$ be an incidence system over $\I$. 

Given a set $A \subseteq V$, define the type of $A$ to be $\type (A) = \lbrace \type (v) : v \in A \rbrace$ and called $\vert \type (A) \vert$ the \textit{rank} of $A$. Also for $A \subseteq V$, define $\vert \I \vert - \vert \type (A) \vert$ to be the \textit{corank} of $A$.

A \textit{flag} is a set of elements $\lbrace v_0,...,v_k \rbrace \subseteq V$ such that for every $0 \leq i,j \leq k$, $v_i * v_j$. A \textit{chamber} is a flag of type $\I$ or equivalently of rank $\vert \I \vert$.

Given a flag $F$, the \textit{residue} of $F$ is the incidence geometry $(F^* \setminus F, *, \type)$ over $\I \setminus \type (F)$ where $F^* = \lbrace v \in V : \forall u \in F, v*u \rbrace$ ($*$, $\type$ are restricted here to $F^* \setminus F$). 

An incidence system $(V,*,\type)$ over $\I$ is called \textit{residually connected} if for every flag $F$ of corank $\geq 2$, the residue of $F$ is connected. 

An incidence system $(V,*,\type)$ over $\I$ is called a \textit{geometry} over $\I$ if every flag is contained in at least one chamber. A \textit{geometry} over $\I$ is called \textit{firm} if every flag of rank $< \vert \I \vert$ is contained in at least two chambers. A geometry over $\I$ is called an \textit{$\I$-geometry} if it is firm and residually connected.

Given an incidence system $(V,*,\type)$ over a finite set $\I$ with $\vert \I \vert = n+1$, we define a simplicial complex $X(V,*,\type)$ as follows: for every $0 \leq k \leq n$, the $k$-simplices of $X (V,*,\type)$ are the flags of $(V,*,\type)$ of rank $k-1$. Note that this is indeed a simplicial complex and moreover this simplicial complex is $(n+1)$-partite. We also note that for a flag $F = \lbrace v_0,...,v_k \rbrace$ in $(V,*,\type)$, $X(F^* \setminus F, *, \type)$ is exactly the link of $\lbrace v_0,...,v_k \rbrace$ in $X (V,*,\type)$. The following table provides a dictionary between the properties of $(V,*,\type)$ and the properties of $X (V,*,\type)$:
\begin{center}
\begin{tabular}{ |m{5cm}|m{5cm}| } 
 \hline
\boldmath$(V,*,\type)$ & \boldmath$X (V,*,\type)$ \\
\hline
\hline
Residually connected & Strongly gallery connected (by Remark \ref{strongly gallery connected to locally connected rmk}) \\
\hline
Geometry over $\I$ & Pure $n$-dimensional \\
\hline
Firm geometry over $\I$& Pure $n$-dimensional and has no free faces \\
\hline
$\I$-geometry & Pure $n$-dimensional, strongly gallery connected and has no free faces \\
 \hline
\end{tabular}
\end{center}

\section{Spectral gap of bipartite graphs and $\varepsilon$-orthogonality}
\label{spectral gap appen}
The aim of this appendix is to prove Theorem \ref{orthogonality-spec thm}. In order to prove this Theorem, we will first gather some facts regarding bipartite graphs and $\varepsilon$-orthogonality.

\subsection{Spectra of random walks on bipartite graphs}
\label{Spectra of random walks on bipartite graphs}

Given a finite graph $(V,E)$, we denote the valency of $v \in V$ as $w(v) = \vert \lbrace u : \lbrace v,u \rbrace \in E \rbrace \vert$. We will always assume that $(V,E)$ has no isolated vertices, i.e., that $w(v) \geq 1$ for every $v \in V$. The Markov matrix (or Markov operator) of the graph is the $\vert V \vert \times \vert V \vert$ matrix with the entries
$$M(v,u) = \begin{cases}
\frac{1}{w(v)} & \lbrace v,u \rbrace \in E \\
0 & \lbrace v,u \rbrace \notin E
\end{cases}.$$
This is a stochastic matrix and therefore its largest eigenvalue is $1$ and for every eigenvalue $\lambda$ of $M$, $\vert \lambda \vert \leq 1$. If $(V,E)$ is connected, then the eigenvalue $1$ has multiplicity $1$ and all the other eigenvalues are strictly less than $1$. The spectral gap of a connected graph $(V,E)$ is $1 -( \text{the second largest eigenvalue of } M)$.

In the case of a biregular, bipartite graph, the spectral gap has an alternative description which we will now explain. Recall that a \textit{bipartite graph} $(V,E)$ is a graph with two disjoint sets of vertices $V_0, V_1 \subset V$, that are called the sides of the graph, such that $V = V_0 \cup V_1$ and every edge in the graph has one vertex in $V_0$ and one vertex is $V_1$, i.e., $\lbrace u, v \rbrace \in E$ implies that $\vert \lbrace u, v \rbrace \cap V_0 \vert = \vert \lbrace u, v \rbrace \cap V_1 \vert =1$. A bipartite graph is called \textit{biregular} or semiregular if there are constants $d_0, d_1$ such that for every $i =0,1$ and every $v \in V_i$, $w(v) = d_i$.

Given a biregular bipartite graph $(V,E)$, the Markov matrix in this case is the matrix
$$M (u,v) = \begin{cases}
\frac{1}{d_0} & \lbrace u , v \rbrace \in E, u \in V_0 \\
\frac{1}{d_1} & \lbrace u , v \rbrace \in E, u \in V_1 \\
0 & \lbrace u , v \rbrace \notin E
\end{cases}.$$
Define the space $\ell_2 (V)$ to be the space of functions $\phi: V \rightarrow \mathbb{C}$ with the inner product
$$\langle \phi , \psi \rangle = d_0 \sum_{v \in V_0} \phi (v) \overline{\psi (v)} + d_1  \sum_{v \in V_1} \phi (v) \overline{\psi (v)}, \forall \phi, \psi \in \ell_2 (V).$$
A direct computation shows that with respect to this norm $M: \ell_2 (V) \rightarrow \ell_2 (V)$ is a self-adjoint operator and therefore has an orthogonal basis of eigenfunctions with real eigenvectors. We note that if $\phi$ is an eigenfunction of $M$ with eigenvalue $\mu$, then
$$\phi' (v) = \begin{cases}
\phi (v) & v \in V_0 \\
- \phi (v) & v \in V_1
\end{cases}$$
is an eigenfunction of $M$ with eigenvalue $-\mu$ (this is shown by direct computation - the details are left for the reader). This implies that the eigenvalues of $M$ are symmetric with respect to $0$ (including the multiplicity), i.e., if $\mu$ is an eigenvalue of $M$ with multiplicity $l$, then $-\mu$ is also an eigenvalue of $M$ with multiplicity $l$. For $U \subseteq V$, denote $\chi_U \in \ell_2 (V)$ to be the indicator function on $U$. We note that $\chi_V$ is an eigenfunction with an eigenvalue $1$ and $\chi_{V_0} - \chi_{V_1}$ is an eigenfunction with an eigenvalue $-1$.

We denote
$$\ell_2^0 (V) = \lbrace \phi \in \ell_2 (V) : \langle \phi, \chi_{V_0} \rangle = \langle \phi, \chi_{V_1} \rangle =0 \rbrace,$$
i.e., $\ell_2^0 (V)$ is the subspace of functions $\phi \in \ell_2 (V)$ such that
$$\sum_{v \in V_0} \phi (v) = \sum_{v \in V_1} \phi (v) =0.$$
We note that $\ell_2^0 (V)$ is exactly all the functions in $\ell_2 (V)$ that are orthogonal to $\chi_V$ and $\chi_{V_0} - \chi_{V_1}$, because $\Span \lbrace \chi_V, \chi_{V_0} - \chi_{V_1} \rbrace = \Span \lbrace \chi_{V_0}, \chi_{V_1} \rbrace$. Recall that eigenvectors of $M$ are orthogonal and therefore
for every $\phi \in \ell_2^0 (V)$, we have that $M \phi \in \ell_2^0 (V)$. This means that restricting $M$ to $\ell_2^0 (V)$ yields $M \vert_{\ell_2^0 (V)} :  \ell_2^0 (V) \rightarrow \ell_2^0 (V)$.
\begin{proposition}
\label{connection - spec gap and M norm prop}
Let $(V,E)$ be a finite, connected, biregular, bipartite graph with more than $2$ vertices and let $\lambda$ be the second largest eigenvalue of $M$. Then $\lambda = \Vert M \vert_{\ell_2^0 (V)} \Vert$, where $\Vert . \Vert$ denotes the operator norm.
\end{proposition}

\begin{proof}
The graph $(V,E)$ has more than two vertices and it is connected and therefore the matrix $M$ has an eigenvalue different that $1$ and $-1$, indeed, recall that if $(V,E)$ is connected that $1$ is an eigenvalue with multiplicity $1$ and since the eigenvalues are symmetric with respect to $0$, $-1$ is an eigenvalue with multiplicity $1$ as well. Therefore $\lambda \geq 0$ and (again using the symmetry of the eigenvalues), every eigenvalue $\mu$ of $M \vert_{\ell_2^0 (V)}$ has $\vert \mu \vert \leq \lambda$. $M$ is self-adjoint and therefore every $\phi \in \ell_2^0 (V)$ can be written as $\phi = \psi_1 +...+\psi_k$ where $\psi_1,...,\psi_k \in \ell_2^0 (V)$ are orthogonal eigenfunctions of $M$ with eigenvalues $\mu_1,...,\mu_k$. Therefore
\begin{dmath*}
\Vert M \phi \Vert^2 = \langle M \phi, M \phi \rangle =
\langle M(\sum_{i=1}^k \psi_i) , M(\sum_{i=1}^k \psi_i) \rangle
=\langle \sum_{i=1}^k \mu_i \psi_i , \sum_{i=1}^k \mu_i \psi_i \rangle  \\
=
\sum_{i=1}^k \mu_i^2 \Vert \psi_i \Vert^2 \leq  \sum_{i=1}^k \lambda^2 \Vert \psi_i \Vert^2 = \lambda \Vert \phi \Vert^2.
\end{dmath*}
The above inequality yields that for every $\phi \in \ell_2^0 (V)$, $\Vert M \phi \Vert \leq \lambda \Vert \phi \Vert$, i.e., $\Vert M \vert_{\ell_2^0 (V)} \Vert \leq \lambda$. To finish the proof we note that the eigenfunction of the eigenvalue $\lambda$ is in $\ell_2^0 (V)$ and therefore $\Vert M \vert_{\ell_2^0 (V)} \Vert =  \lambda$ as needed.
\end{proof}

In continuation to the above proposition, we will show that the second largest eigenvalue of the random walk can be further described as the operator norms of two different operators $M_0$ and $M_1$ defined below. For $i=0,1$, define $\ell_2 (V_i)$ to be the space of functions $\phi: V_i \rightarrow \mathbb{C}$ with the inner product
$$\langle \phi , \psi \rangle = d_i \sum_{v \in V_i} \phi (v) \overline{\psi (v)}, \forall \phi, \psi \in \ell_2 (V_i).$$
Further define the operator $M_i : \ell_2 (V_i) \rightarrow \ell_2 (V_{i + 1})$ (where $i+1$ is taken $\mod 2$) as $M_i = M \vert_{\ell_2 (V_i)}$, i.e., for every $\phi \in \ell_2 (V_i)$ and every $v \in V_{i + 1}$,
$$M_i \phi (v) = \dfrac{1}{d_{i+1}} \sum_{u \in V_i, \lbrace u, v \rbrace \in E} \phi (u).$$
Denote further
$$\ell_2^0 (V_i) = \lbrace \phi \in \ell_2 (V_i) : \sum_{v \in V_i} \phi (v) = 0 \rbrace,$$
and note that restricting $M_i$ to $\ell_2^0 (V_i)$ yields a map $M \vert_{\ell_2^0 (V_i)} : \ell_2^0 (V_i) \rightarrow \ell_2^0 (V_{i+1})$.

\begin{proposition}
\label{connection - spec gap and M_0, M_1 norm prop}
Let $(V,E)$ be a finite, connected, biregular, bipartite graph with more than $2$ vertices and let $\lambda$ be the second largest eigenvalue of $M$. Then $\lambda = \Vert M_0 \vert_{\ell_2^0 (V_0)} \Vert = \Vert M_1 \vert_{\ell_2^0 (V_1)} \Vert$, where $\Vert . \Vert$ denotes the operator norm.
\end{proposition}

\begin{proof}
Note that $M_0$ and $M_1$ are adjoint operators, i.e., $M_0^* = M_1$ (this can be seen either by direct computation or by recalling the fact that $M$ is self-adjoint and $M_i = M \vert_{\ell_2 (V_i)}$). Therefore $(M_0 \vert_{\ell_2^0 (V_0)} )^* = M_1 \vert_{\ell_2^0 (V_1)}$, which yields that
$\Vert M_0 \vert_{\ell_2^0 (V_0)} \Vert = \Vert M_1 \vert_{\ell_2^0 (V_1)} \Vert$ (adjoint operators have the same operator norm).

By Proposition \ref{connection - spec gap and M norm prop}, it is enough to prove that
$$ \Vert M \vert_{\ell_2^0 (V)} \Vert =  \Vert M_0 \vert_{\ell_2^0 (V_0)} \Vert.$$

First we show that
$$\Vert M \vert_{\ell_2^0 (V)} \Vert \geq \Vert M_0 \vert_{\ell_2^0 (V_0)} \Vert.$$
Let $\phi \in \ell_2^0 (V_0)$ with $\Vert \phi \Vert =1$. Define $\psi \in  \ell_2^0 (V)$ as
$$\psi (v) = \begin{cases}
\phi (v) & v \in V_0 \\
0 & v \in V_1
\end{cases}.$$
By this definition $\Vert \psi \Vert =1$ and
$$\Vert M_0 \phi \Vert = \Vert M \psi \Vert \leq \Vert M \vert_{\ell_2^0 (V)} \Vert.$$
This is true for every $\phi \in \ell_2^0 (V_0)$ with $\Vert \phi \Vert =1$ and therefore $\Vert M \vert_{\ell_2^0 (V)} \Vert \geq  \Vert M_0 \vert_{\ell_2^0 (V_0)} \Vert$.

In the other direction, we note that for every $\phi \in \ell_2^0 (V)$, we can define $\phi_i \in \ell_2^0 (V)$ as $\phi_i = \phi \vert_{\ell_2^0 (V_i)}$. We note that $V_0, V_1$ are disjoint and $V_0 \cup V_1 = V$ implies that
$$\Vert \phi \Vert^2 = \Vert \phi_0 \Vert^2+  \Vert \phi_1 \Vert^2.$$
By the definition of $M_0, M_1$,
\begin{dmath*}
\Vert M \phi \Vert^2 = d_1 \sum_{v \in V_1} \Vert (M \phi) (v) \Vert^2 + d_0 \sum_{v \in V_0} \Vert (M \phi (v)) \Vert^2 \\
= d_1 \sum_{v \in V_1} \Vert (M_0 \phi_0) (v) \Vert^2 + d_0 \sum_{v \in V_0} \Vert (M_1 \phi_1 (v)) \Vert^2 \\
\leq \Vert M_0 \vert_{\ell_2^0 (V_0)} \Vert^2 \Vert \phi_0 \Vert^2 +  \Vert M_1 \vert_{\ell_2^0 (V_1)} \Vert^2 \Vert \phi_1 \Vert^2 \\
= \Vert M_0 \vert_{\ell_2^0 (V_0)} \Vert^2 (\Vert \phi_0 \Vert^2 + \Vert \phi_1 \Vert^2) \\
= \Vert M_0 \vert_{\ell_2^0 (V_0)} \Vert^2 \Vert \phi \Vert^2.
\end{dmath*}
Therefore $\Vert M \vert_{\ell_2^0 (V)} \Vert \leq  \Vert M_0 \vert_{\ell_2^0 (V_0)} \Vert$ as needed.
\end{proof}

\subsection{$\varepsilon$-orthogonality}
\label{varepsilon orth section}

We recall the definitions given above regarding $\varepsilon$-orthogonality:
\begin{definition}
\label{varepsilon orthogonal subspaces def}
Let $\mathcal{H}$ be a Hilbert space and let $\mathcal{H}_1, \mathcal{H}_2 \subseteq \mathcal{H}$ be closed subspaces of $\mathcal{H}$. Denote
\begin{dmath*}
\theta (\mathcal{H}_1, \mathcal{H}_2) =
{\inf \lbrace \kappa :  \vert \langle x_1, x_2 \rangle \vert \leq \kappa \Vert x_1 \Vert \Vert x_2 \Vert, \forall x_1 \in \mathcal{H}_1 \cap (\mathcal{H}_1 \cap  \mathcal{H}_2)^\perp, \forall x_2 \in \mathcal{H}_2 \cap (\mathcal{H}_1 \cap  \mathcal{H}_2)^\perp  \rbrace}.
\end{dmath*}
For $0 \leq \varepsilon \leq 1$, $\mathcal{H}_1, \mathcal{H}_2$ will be called $\varepsilon$-orthogonal if   $\theta (\mathcal{H}_1, \mathcal{H}_2) \leq \varepsilon$.
\end{definition}

The definition of $\varepsilon$-orthogonality can be stated using orthogonal projections as follows (for proof see for instance \cite{DeutschBook}[Lemma 9.5]):

\begin{proposition}
\label{orthogonality as projection prop}
Let $\mathcal{H}$ be a Hilbert space and let $\mathcal{H}_1, \mathcal{H}_2 \subseteq \mathcal{H}$ be closed subspaces of $\mathcal{H}$. Denote $P_{\mathcal{H}_1}, P_{\mathcal{H}_2}, P_{\mathcal{H}_1 \cap \mathcal{H}_2}$ to be the orthogonal projections on $\mathcal{H}_1, \mathcal{H}_2, \mathcal{H}_1 \cap \mathcal{H}_2$ respectively. Then $\theta (\mathcal{H}_1, \mathcal{H}_2) = \Vert P_{\mathcal{H}_2} P_{\mathcal{H}_1} - P_{\mathcal{H}_1 \cap \mathcal{H}_2} \Vert = \Vert P_{\mathcal{H}_1} P_{\mathcal{H}_2} - P_{\mathcal{H}_1 \cap \mathcal{H}_2} \Vert$, where $\Vert . \Vert$ denotes the operator norm.

As a result, $\mathcal{H}_1, \mathcal{H}_2$ are $\varepsilon$-orthogonal if and only if $\Vert P_{\mathcal{H}_2} P_{\mathcal{H}_1} - P_{\mathcal{H}_1 \cap \mathcal{H}_2} \Vert \leq \varepsilon$ (or $\Vert P_{\mathcal{H}_1} P_{\mathcal{H}_2} - P_{\mathcal{H}_1 \cap \mathcal{H}_2} \Vert \leq \varepsilon$).
\end{proposition}

\begin{definition}
Let $G$ be a group with subgroups $K_{\lbrace 0 \rbrace}, K_{\lbrace 1 \rbrace} <G$, such that $\langle K_{\lbrace 0 \rbrace}, K_{\lbrace 1 \rbrace} \rangle = G$. For a unitary representation $(\pi, \mathcal{H})$, we denote
$$\mathcal{H}^{\pi (K_{\lbrace 0 \rbrace})} = \lbrace x \in \mathcal{H} : \forall g \in K_{\lbrace 0 \rbrace}, \pi (g) x = x \rbrace,$$
$$\mathcal{H}^{\pi (K_{\lbrace 1 \rbrace})} = \lbrace x \in \mathcal{H} : \forall g \in K_{\lbrace 1 \rbrace}, \pi (g) x = x \rbrace.$$

We define $K_{\lbrace 0 \rbrace}, K_{\lbrace 1 \rbrace}$ to be $\varepsilon$-orthogonal if for every unitary representation $(\pi, \mathcal{H})$ of $G$, $\mathcal{H}^{\pi (K_{\lbrace 0 \rbrace})}$ and $\mathcal{H}^{\pi (K_{\lbrace 1 \rbrace})}$ are $\varepsilon$-orthogonal.
\end{definition}

\begin{proposition}
\label{H_K_0 cap H_K_1 = G prop}
Let $G, K_{\lbrace 0 \rbrace}, K_{\lbrace 1 \rbrace} <G$ such that $\langle K_{\lbrace 0 \rbrace}, K_{\lbrace 1 \rbrace} \rangle = G$ and $(\pi, \mathcal{H})$ a unitary representation of $G$. Denote
$$\mathcal{H}^{\pi (G)} = \lbrace x \in \mathcal{H} : \forall g \in G, \pi (g) x = x \rbrace.$$
Then $\mathcal{H}^{\pi (G)} = \mathcal{H}^{\pi (K_{\lbrace 0 \rbrace})} \cap \mathcal{H}^{\pi (K_{\lbrace 1 \rbrace})}$. 
\end{proposition}

\begin{proof}
It is obvious that $\mathcal{H}^{\pi (G)} \subseteq \mathcal{H}^{\pi (K_{\lbrace 0 \rbrace})} \cap \mathcal{H}^{\pi (K_{\lbrace 1 \rbrace})}$. In the other direction, we will show that for every $x \in \mathcal{H}^{\pi (K_{\lbrace 0 \rbrace})} \cap \mathcal{H}^{\pi (K_{\lbrace 1 \rbrace})}$ and $g \in G$, $g.x = x$. Fix such $x$ and $g$. By our assumption, $\langle K_{\lbrace 0 \rbrace}, K_{\lbrace 1 \rbrace} \rangle = G$, thus there are $h_1,...,h_n \in K_{\lbrace 0 \rbrace} \cup K_{\lbrace 1 \rbrace}$ such that $g=h_1...h_n$. Then,
$$\pi (g) x = \pi (h_1 ... h_n) x = \pi (h_1) ... \pi (h_n) x = \pi (h_1) ... \pi (h_{n-1}) x =...=x,$$
i.e., $x \in \mathcal{H}^{\pi (G)}$ as needed.
\end{proof}

We saw above that the $\varepsilon$-orthogonality between subspaces can be described in terms of orthogonal projections. In the case where $G$ is finite, these projections can be defined explicitly using the group algebra. Define the following elements in the group algebra $\mathbb{C}[G]$:
$$k_{ \lbrace 0 \rbrace} = \sum_{g \in K_{\lbrace 0 \rbrace}} \dfrac{1}{\vert K_{\lbrace 0 \rbrace} \vert} g,$$
$$k_{ \lbrace 1 \rbrace} = \sum_{g \in K_{\lbrace 1 \rbrace}} \dfrac{1}{\vert K_{\lbrace 1 \rbrace} \vert} g,$$
$$k_{G} = \sum_{g \in G} \dfrac{1}{\vert G \vert} g.$$

\begin{proposition}
\label{alg property of k's prop}
Let $G$ be a finite group and $K_{ \lbrace 0 \rbrace}$, $K_{ \lbrace 1 \rbrace}$, $k_{ \lbrace 0 \rbrace}$, $k_{ \lbrace 1 \rbrace}$ as above. Then
\begin{enumerate}
\item For every $i=0,1$ and every $g \in K_{\lbrace i \rbrace}$, we have $g k_{ \lbrace i \rbrace} =  k_{ \lbrace i \rbrace} g = k_{ \lbrace i \rbrace}$. Also, for every $g \in G$, $g k_G = k_G g = k_G$.
\item For every $i=0,1$, $(k_{ \lbrace i \rbrace})^2 = (k_{ \lbrace i \rbrace})^* = k_{ \lbrace i \rbrace}$. Also, $(k_G)^2 = (k_G)^* = k_G$.
%\item For every $i=0,1$, $k_{ \lbrace i \rbrace} k_G = k_G k_{ \lbrace i \rbrace} = k_G$.
%\item For every $i=0,1$, $(k_{ \lbrace i \rbrace} - k_G )^2 = (k_{ \lbrace i \rbrace} - k_G )^* = k_{ \lbrace i \rbrace} - k_G$.
\end{enumerate}
\end{proposition}

\begin{proof}
We will prove the assertions for $k_{ \lbrace 0 \rbrace}$, the proofs for $k_{ \lbrace 1 \rbrace}$ and $k_G$ are similar. Let $g \in K_{\lbrace 0 \rbrace}$, then
$$g k_{ \lbrace 0 \rbrace} = \sum_{g' \in K_{\lbrace 0 \rbrace}} \dfrac{1}{\vert K_{\lbrace 0 \rbrace} \vert} g g' = \sum_{g' \in g^{-1} K_{\lbrace 0 \rbrace}} \dfrac{1}{\vert K_{\lbrace 0 \rbrace} \vert} g' = \sum_{g' \in K_{\lbrace 0 \rbrace}} \dfrac{1}{\vert K_{\lbrace 0 \rbrace} \vert} g' = k_{ \lbrace 0 \rbrace} .$$
Similarly,
$$ k_{ \lbrace 0 \rbrace} g = \sum_{g' \in K_{\lbrace 0 \rbrace}} \dfrac{1}{\vert K_{\lbrace 0 \rbrace} \vert} g' g = \sum_{g' \in  K_{\lbrace 0 \rbrace} g^{-1}} \dfrac{1}{\vert K_{\lbrace 0 \rbrace} \vert} g' = \sum_{g' \in K_{\lbrace 0 \rbrace}} \dfrac{1}{\vert K_{\lbrace 0 \rbrace} \vert} g' = k_{ \lbrace 0 \rbrace} .$$

For the second assertion, we have that by the first assertion
$$(k_{ \lbrace 0 \rbrace})^2 = \sum_{g \in K_{\lbrace 0 \rbrace}} \dfrac{1}{\vert K_{\lbrace 0 \rbrace} \vert} (g k_{ \lbrace 0 \rbrace}) = \dfrac{1}{\vert K_{\lbrace 0 \rbrace} \vert} \vert K_{\lbrace 0 \rbrace} \vert k_{ \lbrace 0 \rbrace} = k_{ \lbrace 0 \rbrace}.$$
Also,
$$(k_{ \lbrace 0 \rbrace})^* = \sum_{g \in K_{\lbrace 0 \rbrace}} \overline{ \dfrac{1}{\vert K_{\lbrace 0 \rbrace} \vert} } (g^{-1}) =  \sum_{g \in K_{\lbrace 0 \rbrace}}  \dfrac{1}{\vert K_{\lbrace 0 \rbrace} \vert} (g^{-1}) = k_{ \lbrace 0 \rbrace}.$$

%For the third assertion, note that for every $g \in G$, $g k_G = k_G g = k_G$ and therefore for $i=0,1$,
%$$k_{ \lbrace i \rbrace} k_G = \sum_{g \in K_{\lbrace i \rbrace}} \dfrac{1}{\vert K_{\lbrace i \rbrace} \vert} g k_G = \sum_{g \in K_{\lbrace i \rbrace}} \dfrac{1}{\vert K_{\lbrace i \rbrace} \vert} \vert K_{\lbrace i \rbrace} \vert k_G = k_G,$$
%and similarly $k_G k_{ \lbrace i \rbrace} = k_G$.

%Last, we note that by the second assertion, for $i=0,1$ we have that
%$$(k_{ \lbrace i \rbrace} - k_G )^* = (k_{ \lbrace i \rbrace})^* - (k_G )^* = k_{ \lbrace i \rbrace} - k_G.$$
%Also, by the third assertion
%$$(k_{ \lbrace i \rbrace} - k_G )^2 = (k_{ \lbrace i \rbrace})^2 - k_{ \lbrace i \rbrace} k_G - k_G k_{ \lbrace i \rbrace} + (k_G)^2 = k_{ \lbrace i \rbrace} - k_G - k_G + k_G = k_{ \lbrace i \rbrace} - k_G.$$
\end{proof}

\begin{corollary}
\label{pi(k)'s are projection coro}
Let $G$ be a finite group with subgroups $K_{\lbrace 0 \rbrace}$, $K_{\lbrace 1 \rbrace}$ as above and $(\pi, \mathcal{H})$ be a unitary representation of $G$. Then $\pi (k_G),\pi (k_{\lbrace 0 \rbrace}), \pi (k_{\lbrace 1 \rbrace})$ are orthogonal projections on $\mathcal{H}^{\pi(G)}, \mathcal{H}^{\pi(K_{\lbrace 0 \rbrace})},  \mathcal{H}^{\pi(K_{\lbrace 1 \rbrace})}$ respectively.

%Also, $\pi (k_{\lbrace 0 \rbrace} -k_G), \pi (k_{\lbrace 1 \rbrace} -k_G)$ are orthogonal projections on $\mathcal{H}^{K_{\lbrace 0 \rbrace}} \cap (\mathcal{H}^{G})^\perp,  \mathcal{H}^{K_{\lbrace 1 \rbrace}} \cap (\mathcal{H}^{G})^\perp$ respectively.
\end{corollary}

\begin{proof}
By Proposition \ref{alg property of k's prop}, for every $i=0,1$,
$$(\pi (k_{ \lbrace i \rbrace}))^2 = (\pi (k_{ \lbrace i \rbrace}))^* = \pi (k_{ \lbrace i \rbrace}).$$
%$$(\pi (k_{ \lbrace i \rbrace} - k_G))^2 = (\pi (k_{ \lbrace i \rbrace} - k_G))^* = \pi (k_{ \lbrace i \rbrace} -k_G).$$
Also, $(\pi (k_G))^2 = ( \pi (k_G))^* = \pi (k_G)$. Therefore $\pi (k_{ \lbrace i \rbrace})$, $i=0,1$ and $\pi (k_G)$ are orthogonal projections.

We will only show that $\pi (k_{\lbrace 0 \rbrace})$ is an orthogonal projection on $\mathcal{H}^{\pi (K_{\lbrace 0 \rbrace})}$ (the cases of  $\pi (k_{\lbrace 1 \rbrace})$ and $\pi (k_G)$ are similar). We note that for every $x \in \mathcal{H}^{\pi (K_{\lbrace 0 \rbrace})}$,
$$\pi (k_{\lbrace 0 \rbrace}) x = \sum_{g \in K_{\lbrace 0 \rbrace}} \frac{1}{\vert K_{\lbrace 0 \rbrace} \vert} \pi (g)x =  \sum_{g \in K_{\lbrace 0 \rbrace}} \frac{1}{\vert K_{\lbrace 0 \rbrace} \vert} x = x,$$
and therefore $\im (\pi (k_{\lbrace 0 \rbrace})) \subseteq \mathcal{H}^{\pi (K_{\lbrace 0 \rbrace})}$. In the other direction, we note that for every $x \in \im (\pi (k_{\lbrace 0 \rbrace}))$ we have that $x = \pi (k_{\lbrace 0 \rbrace}) x$. Therefore for every $g \in K_{\lbrace 0 \rbrace}$, using Proposition \ref{alg property of k's prop},
$$\pi (g) x= \pi (g) \pi (k_{\lbrace 0 \rbrace}) x = \pi (g k_{\lbrace 0 \rbrace}) x = \pi ( k_{\lbrace 0 \rbrace}) x = x.$$
Therefore $\mathcal{H}^{\pi (K_{\lbrace 0 \rbrace})} \subseteq \im (\pi (k_{\lbrace 0 \rbrace}))$ as needed.
\end{proof}

The above discussion leads to the following corollary:

\begin{corollary}
\label{norm condition of epsilon orthogonality coro}
Let $G$ be a finite group  with subgroups $K_{ \lbrace 0 \rbrace}$, $K_{ \lbrace 1 \rbrace}$ as above. Then for any unitary representation $(\pi, \mathcal{H})$, the subspaces $\mathcal{H}^{\pi (K_{\lbrace 0 \rbrace})}$ and $\mathcal{H}^{\pi (K_{\lbrace 1 \rbrace})}$ are $\varepsilon$-orthogonal if and only if $\Vert \pi (k_{ \lbrace 1 \rbrace} k_{ \lbrace 0 \rbrace} - k_G) \Vert \leq \varepsilon$.
\end{corollary}

\begin{proof}
Combine Corollary \ref{pi(k)'s are projection coro}, Proposition \ref{H_K_0 cap H_K_1 = G prop} and Proposition \ref{orthogonality as projection prop}.
\end{proof}

Finally, by Peter-Weyl Theorem, it is enough to consider the right (or left) regular representation of $G$:
\begin{proposition}
\label{reg. rep. condition of epsilon orthogonality prop}
Let $G$ be a finite group  with subgroups $K_{ \lbrace 0 \rbrace}$, $K_{ \lbrace 1 \rbrace}$ as above. Then for any unitary representation $(\pi, \mathcal{H})$, $\mathcal{H}^{\pi (K_{\lbrace 0 \rbrace})}$ and $\mathcal{H}^{\pi (K_{\lbrace 1 \rbrace})}$ are $\varepsilon$-orthogonal if and only if $\Vert \rho (k_{ \lbrace 1 \rbrace} k_{ \lbrace 0 \rbrace} - k_G) \Vert \leq \varepsilon$, where $\rho$ is the right regular representation of $G$.
\end{proposition} 

\begin{proof}
First, by Peter-Weyl Theorem every unitary representation can be decomposed into an orthogonal sum of irreducible unitary representations and therefore it is enough to check the $\varepsilon$-orthogonality condition in each component of this orthogonal sum, i.e., to check $\varepsilon$-orthogonality only for the irreducible unitary representations.

Second, also by Peter-Weyl Theorem the right regular representation decomposes as an orthogonal sum that contains all the irreducible unitary representations and therefore if $\ell_2 (G)^{\pi (K_{\lbrace 0 \rbrace})}$ and $\ell_2 (G)^{\pi (K_{\lbrace 1 \rbrace})}$ are $\varepsilon$-orthogonal, then for every irreducible unitary representation $(\pi, \mathcal{H})$, $\mathcal{H}^{\pi (K_{\lbrace 0 \rbrace})}$ and $\mathcal{H}^{\pi (K_{\lbrace 1 \rbrace})}$ are $\varepsilon$-orthogonal. 

Combining these facts with Corollary \ref{norm condition of epsilon orthogonality coro} finishes the proof.
\end{proof}

\subsection{A bound on the spectrum of X via $\varepsilon$-orthogonality}
\label{The second largest eigenvalue of the random walk of X_G via varepsilon-orthogonality sect}

\begin{comment}
\begin{theorem}
Let $G$ be a finite group with a subgroup geometry system $(G, ( K_{\lbrace 0 \rbrace}, K_{\lbrace 1 \rbrace}))$. Also, let $\lambda$ be the second largest eigenvalue of the simple random walk on $X(G, ( K_{\lbrace 0 \rbrace}, K_{\lbrace 1 \rbrace}))$. For every $0 \leq \varepsilon$, $\lambda \leq \varepsilon$ if and only if $K_{ \lbrace 0 \rbrace}$ and $K_{ \lbrace 1 \rbrace}$ are $\varepsilon$-orthogonal.
\end{theorem}
\end{comment}

Let $G$ be a group and $(G, ( K_{\lbrace 0 \rbrace}, K_{\lbrace 1 \rbrace}))$ be a subgroup geometry system. For the reader's convenience, we repeat the construction of $X(G, ( K_{\lbrace 0 \rbrace}, K_{\lbrace 1 \rbrace}))$: the vertices of $X_G$ are $V = V_0 \cup V_1$ with
$$V_i = \lbrace g K_{\lbrace i \rbrace} : g \in G \rbrace, i=0,1,$$
(in particular, $\vert V_i \vert = \vert G/ K_{\lbrace i \rbrace} \vert$). The edges $E$ of $X_G$ are
$$E= \lbrace g K_{\lbrace 0,1 \rbrace} : g \in G \rbrace.$$
and $g K_{\lbrace 0,1 \rbrace} \in E$ is the edge connecting $g K_{\lbrace 0 \rbrace} \in V_0$ and $g K_{\lbrace 1 \rbrace} \in V_1$. 
%This means that $X_G$ is a connected biregular bipartite graph: the number of edges connected to each vertex in $V_i$ is $d_i = \vert K_{\lbrace i \rbrace} / K_{\lbrace 0,1 \rbrace} \vert$ for $i=0,1$.

\begin{proposition}
\label{incidence in cosets prop}
Let $G$ be a group and $(G, ( K_{\lbrace 0 \rbrace}, K_{\lbrace 1 \rbrace}))$ be a subgroup geometry system. Denote as above
$$V_i = \lbrace g K_{\lbrace i \rbrace} : g \in G \rbrace, i=0,1.$$

If $[K_{\lbrace 0 \rbrace} :K_{\lbrace 0,1 \rbrace}] < \infty , [K_{\lbrace 1 \rbrace} :K_{\lbrace 0,1 \rbrace}] < \infty$, then $X(G, ( K_{\lbrace 0 \rbrace}, K_{\lbrace 1 \rbrace}))$ is a connected locally finite biregular bipartite graph:  the number of edges connected to each vertex in $V_i$ is $d_i = [K_{\lbrace i \rbrace} :K_{\lbrace 0,1 \rbrace}]$ for $i=0,1$.
Moreover, for $i=0,1$, choose $h_1^i,...,h_{d_i}^i \in K_{\lbrace i \rbrace} / K_{\lbrace 0,1 \rbrace}$ such that $K_{\lbrace i \rbrace} = \bigcup_{j=1}^{d_i} h_j^i K_{\lbrace 0,1 \rbrace}$. Then:
\begin{enumerate}
\item For every $g K_{\lbrace 0 \rbrace} \in V_0$, $g K_{\lbrace 0 \rbrace}$ is connected by an edge to $g h_j^0 K_{\lbrace 1 \rbrace} \in V_1$ for $j=1,...,d_0$ and only to these vertices.
\item For every $g K_{\lbrace 1 \rbrace} \in V_1$, $g K_{\lbrace 1 \rbrace}$ is connected by an edge to $g h_j^1 K_{\lbrace 0 \rbrace} \in V_0$ for $j=1,...,d_1$ and only to these vertices.
\end{enumerate}
\end{proposition}

\begin{proof}
Both claims are symmetric and therefore it is sufficient to prove only the first one. Let $g K_{\lbrace 0 \rbrace} \in V_0$. Recall that $g K_{\lbrace 0 \rbrace} \leq g' K_{\lbrace 0,1 \rbrace}$ if and only if $g^{-1} g' \in K_{\lbrace 0 \rbrace}$. Therefore $g K_{\lbrace 0 \rbrace} \leq g' K_{\lbrace 0,1 \rbrace}$ if and only if
$$g^{-1} g' \in \bigcup_{j=1}^{d_0} h_j^0 K_{\lbrace 0,1 \rbrace},$$
i.e., if and only if there is $h_{j_0}^0$ such that $g^{-1} g' \in h_{j_0}^0 K_{\lbrace 0,1 \rbrace} \Leftrightarrow g' \in g h_{j_0}^0 K_{\lbrace 0,1 \rbrace}$. Note that for every $g'' \in  K_{\lbrace 0,1 \rbrace}$ if $g' = g h_{j_0}^0 g''$, then $g' K_{\lbrace 0,1 \rbrace} =  g h_{j_0}^0 g'' K_{\lbrace 0,1 \rbrace} =  g h_{j_0}^0 K_{\lbrace 0,1 \rbrace}$, therefore we can always assume that $g' = g h_{j_0}^0$. To conclude, $g K_{\lbrace 0 \rbrace}$ is connected to (and only to) the edges $g h_{1}^0 K_{\lbrace 0,1 \rbrace},...,g h_{d_1}^0 K_{\lbrace 0,1 \rbrace}$ and therefore it is connected by an edge to the vertices $g h_j^0 K_{\lbrace 1 \rbrace} \in V_1$ for $j=1,...,d_0$ and only to these vertices.
\end{proof}

By Proposition \ref{connection - spec gap and M_0, M_1 norm prop}, the second largest eigenvalue of the simple random walk on $X(G, ( K_{\lbrace 0 \rbrace}, K_{\lbrace 1 \rbrace}))$ can be bounded by bounding the norm of $M_0 \vert_{\ell_2^0 (V_0)}$. We will show that $M_0 \vert_{\ell_2^0 (V_0)}$ can be bounded using the notion of $\varepsilon$-orthogonality between the groups $K_{\lbrace 0 \rbrace}$ and $K_{\lbrace 1 \rbrace}$ defined above.

\begin{theorem}
\label{Spectral gap as orthogonality thm}
Let $G$ be a finite group with subgroups $K_{ \lbrace 0 \rbrace}, K_{ \lbrace 1 \rbrace}$, such that $\langle K_{ \lbrace 0 \rbrace}, K_{ \lbrace 1 \rbrace} \rangle =G$. For every $0 \leq \varepsilon$, $\Vert M_0 \vert_{\ell_2^0 (V_0)} \Vert \leq \varepsilon$ if and only if $K_{ \lbrace 0 \rbrace}$ and $K_{ \lbrace 1 \rbrace}$ are $\varepsilon$-orthogonal.
\end{theorem}

In order to prove this theorem, we will need the following lemma:
\begin{lemma}
\label{ell_2^0 (G) lemma}
Let $G$ be a finite group with subgroups $K_{ \lbrace 0 \rbrace}, K_{ \lbrace 1 \rbrace}$, such that $\langle K_{ \lbrace 0 \rbrace}, K_{ \lbrace 1 \rbrace} \rangle =G$. Define the following subspaces of $\ell_2 (G)$:
$$\ell_2^0 (G) = \lbrace \phi \in \ell_2 (G) : \sum_{g \in G} \phi (g) =0 \rbrace,$$
For $i=0,1$,
$$\ell_2 (G)_{\lbrace i \rbrace} = \lbrace \phi \in \ell_2 (G) : \forall g_1, g_2 \in G, g_1 K_{\lbrace i \rbrace} = g_2 K_{\lbrace i \rbrace} \Rightarrow \phi (g_1) = \phi (g_2) \rbrace,$$
$$\ell_2^0 (G)_{\lbrace i \rbrace} = \lbrace \phi \in \ell_2 (G)_{\lbrace i \rbrace} : \sum_{g \in G} \phi (g) =0 \rbrace.$$
Then:
\begin{enumerate}
\item $I-\rho (k_G)$ is the orthogonal projection on $\ell_2^0 (G)$.
\item For $i=0,1$, $\rho (k_{\lbrace i \rbrace})$ is the orthogonal projection on $\ell_2 (G)_{\lbrace i \rbrace}$.
\item For $i=0,1$, $\rho (k_{\lbrace i \rbrace}-k_G)$ is the orthogonal projection on $\ell_2^0 (G)_{\lbrace i \rbrace}$.
\end{enumerate}
\end{lemma}

\begin{proof}
For the first assertion, we note that for every $\phi \in \ell_2 (G)$ and every $g_0 \in G$,
\begin{dmath*}
(\rho (k_G) \phi) (g_0) = \dfrac{1}{\vert G \vert} \sum_{g \in G} \phi (g_0 g) = \dfrac{1}{\vert G \vert} \sum_{g \in G} \phi (g).
\end{dmath*}
Therefore $\rho (k_G)$ is a projection on the space of constant functions and $I-\rho(k_G)$ is a projection on the space orthogonal to the space of constant functions, i.e., on $\ell_2^0 (G)$.

We will prove the second and third assertions for the case $i=0$ (the proof in the case $i=1$ is similar). For every $g_1, g_2 \in G$ such that $g_1 K_{\lbrace 0 \rbrace} = g_2 K_{\lbrace 0 \rbrace}$, we have that $g_2^{-1} g_1 \in K_{\lbrace 0 \rbrace}$ which implies that $g_2^{-1} g_1 K_{\lbrace 0 \rbrace} = K_{\lbrace 0 \rbrace}$. Therefore for every $\phi \in \ell_2 (G)$ we have that
\begin{dmath*}
(\rho (k_{\lbrace 0 \rbrace}) \phi) (g_1) =
\dfrac{1}{\vert K_{\lbrace 0 \rbrace} \vert} \sum_{g \in K_{\lbrace 0 \rbrace}} \phi (g_1 g) =
\dfrac{1}{\vert K_{\lbrace 0 \rbrace} \vert} \sum_{g \in g_2^{-1} g_1 K_{\lbrace 0 \rbrace}} \phi (g_1 g) =
\dfrac{1}{\vert K_{\lbrace 0 \rbrace} \vert} \sum_{g \in K_{\lbrace 0 \rbrace}} \phi (g_1 g_1^{-1} g_2 g) =
\dfrac{1}{\vert K_{\lbrace 0 \rbrace} \vert} \sum_{g \in K_{\lbrace 0 \rbrace}} \phi (g_2 g) =
(\rho (k_{\lbrace 0 \rbrace}) \phi) (g_2),
\end{dmath*}
thus, $\im (\rho (k_{\lbrace 0 \rbrace})) \subseteq \ell_2 (G)_{\lbrace 0 \rbrace}$. In the other direction, let $\phi \in \ell_2 (G)_{\lbrace 0 \rbrace}$, and note that for every $g \in K_{\lbrace 0 \rbrace}$ and every $g' \in G$, we have that $g' g  K_{\lbrace 0 \rbrace} =  g' K_{\lbrace 0 \rbrace}$ and thus $\phi (g') = \phi (g' g)$. It follows that for every $g' \in G$, 
$$(\rho (k_{\lbrace 0 \rbrace}) \phi) (g') = \dfrac{1}{\vert K_{\lbrace 0 \rbrace} \vert} \sum_{g \in K_{\lbrace 0 \rbrace}} \phi (g' g) = \dfrac{1}{\vert K_{\lbrace 0 \rbrace} \vert} \sum_{g \in K_{\lbrace 0 \rbrace}} \phi (g') = \phi (g'),$$
i.e., $\rho (k_{\lbrace 0 \rbrace}) \phi = \phi$ and thus $\im (\rho (k_{\lbrace 0 \rbrace})) \supseteq \ell_2 (G)_{\lbrace 0 \rbrace}$.

We note that $\rho (k_{\lbrace i \rbrace}-k_G) = \rho ((k_{\lbrace i \rbrace}) (I- \rho (k_G)) = (I- \rho (k_G)) \rho (k_{\lbrace i \rbrace})$. This means that $\rho (k_{\lbrace i \rbrace}-k_G)$ is the product of commuting two projections: $\rho (k_{\lbrace i \rbrace})$ and $I- \rho (k_G)$. Recall that for two commuting projections, the image of the product is the intersection of the images and thus in our case $\im (\rho (k_{\lbrace i \rbrace}-k_G)) = \ell_2 (G)_{\lbrace i \rbrace} \cap \ell_2^0 (G) = \ell_2^0 (G)_{\lbrace i \rbrace}$ as needed.
\end{proof}

After this lemma, we turn to prove Theorem \ref{Spectral gap as orthogonality thm}:
\begin{proof}[Proof of Theorem \ref{Spectral gap as orthogonality thm}]
By Proposition \ref{reg. rep. condition of epsilon orthogonality prop}, in order to prove the theorem it is enough to prove that for the right regular representation $\rho$, $\Vert M_0 \vert_{\ell_2^0 (V_0)} \Vert = \Vert \rho (k_{\lbrace 1 \rbrace} k_{\lbrace 0 \rbrace} - k_G) \Vert$.

\textbf{Step 1:} Using the notations of Lemma \ref{ell_2^0 (G) lemma}, we define maps
$$L_i : \ell_2^0 (V_i) \rightarrow \ell_2^0 (G)_{\lbrace i \rbrace},$$
as follows: we recall that $V_i = \lbrace g K_{\lbrace i \rbrace} : g \in G \rbrace$ and therefore, for every $\phi \in \ell_2^0 (V_i)$ we define for every $g \in G$
$$(L_i \phi) (g) = \frac{1}{\vert K_{\lbrace 0,1 \rbrace} \vert}\phi (g K_{\lbrace i \rbrace}).$$
We claim that $L_i$ is an isometric isomorphism of $\ell_2^0 (V_i)$ and $\ell_2^0 (G)_{\lbrace i \rbrace}$. Indeed, $L_i$ is invertible and its inverse is
$$(L_i^{-1} \phi) (g K_{\lbrace i \rbrace}) = \vert K_{\lbrace 0,1 \rbrace} \vert \phi (g), \forall \phi \in \ell_2^0 (G)_{\lbrace i \rbrace},$$
(this is well defined because of the definition of $\ell_2^0 (G)_{\lbrace i \rbrace}$). To see that $L_i$ is an isometry, we note that for every $\phi \in \ell_2^0 (V_i)$,
\begin{dmath*}
\Vert L_i \phi \Vert_{\ell_2 (G)} =
\frac{1}{\vert K_{\lbrace 0,1 \rbrace} \vert} \sum_{g \in G} \vert L_i \phi (g) \vert^2 = \frac{1}{\vert K_{\lbrace 0,1 \rbrace} \vert}  \sum_{g \in G / K_{\lbrace i \rbrace} } \vert K_{\lbrace i \rbrace} \vert  \vert \phi (g K_{\lbrace i \rbrace}) \vert^2 = \\
 \sum_{g \in G / K_{\lbrace i \rbrace} } \dfrac{\vert K_{\lbrace i \rbrace} \vert}{\vert K_{\lbrace 0,1 \rbrace} \vert}  \vert \phi (g K_{\lbrace i \rbrace}) \vert^2 =
 \sum_{g \in G / K_{\lbrace i \rbrace} } d_i  \vert \phi (g K_{\lbrace i \rbrace}) \vert^2 =
\Vert \phi \Vert^2_{\ell_2 (V_{i})}.
\end{dmath*}

\textbf{Step 2:} We will prove that $\Vert \rho (k_{\lbrace 1 \rbrace} k_{\lbrace 0 \rbrace} - k_G) \Vert = \Vert \rho (k_{\lbrace 1 \rbrace}) \vert_{\ell_2^0 (G)_{\lbrace 0 \rbrace}} \Vert$.

In the case where $\rho (k_{\lbrace 0 \rbrace} - k_G) x =0$ for every $x \in \ell_2 (G)$, we have that for every $x$, $\rho (k_{\lbrace 1 \rbrace} k_{\lbrace 0 \rbrace} - k_G) x =0$ and therefore $\Vert \rho (k_{\lbrace 1 \rbrace} k_{\lbrace 0 \rbrace} - k_G) \Vert =0$. In this case, $\ell_2^0 (G)_{\lbrace 0 \rbrace} = \lbrace 0 \rbrace$ and therefore $\Vert \rho (k_{\lbrace 1 \rbrace}) \vert_{\ell_2^0 (G)_{\lbrace 0 \rbrace}} \Vert =0$ and the equality is proved.

Assume now that there is $x \in \ell_2 (G)$ such that $\rho (k_{\lbrace 0 \rbrace} - k_G) x  \neq 0$. For every such $x$, recall that $\Vert \rho (k_{\lbrace 0 \rbrace} - k_G) x \Vert \leq \Vert \rho (k_{\lbrace 0 \rbrace} - k_G) \Vert \Vert x \Vert \leq \Vert x \Vert$. Therefore
\begin{dmath*}
\dfrac{\Vert \rho (k_{\lbrace 1 \rbrace} k_{\lbrace 0 \rbrace} - k_G) x \Vert }{\Vert x \Vert} \leq  \dfrac{\Vert \rho (k_{\lbrace 1 \rbrace}) \rho (k_{\lbrace 0 \rbrace} - k_G) x \Vert }{\Vert \rho (k_{\lbrace 0 \rbrace} - k_G) x \Vert} \leq \\
\sup_{y \in \ell_2^0 (G)_{\lbrace 0 \rbrace}, y \neq 0} \dfrac{\Vert \rho (k_{\lbrace 1 \rbrace}) y \Vert}{\Vert y \Vert} =
 \Vert \rho (k_{\lbrace 1 \rbrace}) \vert_{\ell_2^0 (G)_{\lbrace 0 \rbrace}} \Vert,
\end{dmath*}
where the second inequality is due to Lemma \ref{ell_2^0 (G) lemma} in which we proved that $\rho (k_{\lbrace 0 \rbrace} - k_G)$ is the orthogonal projection on $\ell_2^0 (G)_{\lbrace 0 \rbrace}$.
This yields that
\begin{dmath*}
\Vert \rho (k_{\lbrace 1 \rbrace} k_{\lbrace 0 \rbrace} - k_G) \Vert = \\
\sup_{x \in \ell_2 (G), \rho (k_{\lbrace 0 \rbrace} - k_G) x  \neq 0} \dfrac{\Vert \rho (k_{\lbrace 1 \rbrace} k_{\lbrace 0 \rbrace} - k_G) x \Vert }{\Vert x \Vert} \leq \Vert \rho (k_{\lbrace 1 \rbrace}) \vert_{\ell_2^0 (G)_{\lbrace 0 \rbrace}} \Vert.
\end{dmath*}
In the other direction, using Lemma \ref{ell_2^0 (G) lemma} again, we get that for every $x \in \ell_2^0 (G)_{\lbrace 0 \rbrace}$,
\begin{dmath*}
\rho (k_{\lbrace 1 \rbrace} k_{\lbrace 0 \rbrace} - k_G) x = \rho (k_{\lbrace 1 \rbrace}) \rho ( k_{\lbrace 0 \rbrace} - k_G) x
=(\rho (k_{\lbrace 1 \rbrace} ) \vert_{\ell_2^0 (G)_{\lbrace 0 \rbrace}}) x,
\end{dmath*}
and therefore $\Vert \rho (k_{\lbrace 1 \rbrace} k_{\lbrace 0 \rbrace} - k_G) \Vert \geq \Vert \rho (k_{\lbrace 1 \rbrace}) \vert_{\ell_2^0 (G)_{\lbrace 0 \rbrace}} \Vert$ as needed.

\textbf{Step 3:} We will prove that $M_0 \vert_{\ell_2^0 (V_0)} = L_1^{-1} (\rho (k_{\lbrace 1 \rbrace}) \vert_{\ell_2^0 (G)_{\lbrace 0 \rbrace}} ) L_0$. Proving this equality will finish the proof because the $L_0, L_1$ are isometric isomorphisms and therefore this equality implies that $\Vert M_0 \vert_{\ell_2^0 (V_0)} \Vert = \Vert \rho (k_{\lbrace 1 \rbrace}) \vert_{\ell_2^0 (G)_{\lbrace 0 \rbrace}} \Vert$ and this in turn implies that $\Vert M_0 \vert_{\ell_2^0 (V_0)} \Vert =  \Vert \rho (k_{\lbrace 1 \rbrace} k_{\lbrace 0 \rbrace} - k_G) \Vert$ due to the previous step.

In the following computations, in order to ease the reading, we will drop the restriction notation and write $\rho (k_{\lbrace 1 \rbrace})$ instead of $\rho (k_{\lbrace 1 \rbrace}) \vert_{\ell_2^0 (G)_{\lbrace 0 \rbrace}}$.

We note that for every $\phi \in \ell_2^0 (G)_{\lbrace 0 \rbrace}$ and every $g \in G$, $g' \in K_{\lbrace 0,1 \rbrace}$, we have that $\phi (g g') = \phi (g)$ (since $g^{-1} g g' = g' \in K_{\lbrace 0,1 \rbrace} \subseteq K_{\lbrace 0 \rbrace}$). Fix $h_1^1,...,h_{d_1}^1 \in K_{\lbrace 1 \rbrace} / K_{\lbrace 0,1 \rbrace}$ such that $K_{\lbrace 1 \rbrace} = \bigcup_{i=1}^{d_1} h_i^1 K_{\lbrace 0,1 \rbrace}$. Then for every $\phi \in \ell_2^0 (V_0)$ and every $g_0 \in G$,
\begin{dmath*}
(\rho (k_{\lbrace 1 \rbrace}) L_0 \phi) (g_0) = \dfrac{1}{\vert K_{\lbrace 1 \rbrace} \vert} \sum_{g \in K_{\lbrace 1 \rbrace}} (L_0 \phi) (g_0 g) = \dfrac{1}{\vert K_{\lbrace 1 \rbrace} \vert} \sum_{i=1}^{d_1} \sum_{g \in K_{\lbrace 0,1 \rbrace}} (L_0 \phi) (g_0 h_i^1 g) =
\dfrac{1}{\vert K_{\lbrace 1 \rbrace} \vert} \sum_{i=1}^{d_1} \vert K_{\lbrace 0,1 \rbrace} \vert (L_0 \phi) (g_0 h_i^1) =
\dfrac{\vert K_{\lbrace 0,1 \rbrace} \vert}{\vert K_{\lbrace 1 \rbrace} \vert} \sum_{i=1}^{d_1}  (L_0 \phi) (g_0 h_i^1) =
\dfrac{1}{d_1} \sum_{i=1}^{d_1} (L_0 \phi) (g_0 h_i^1) =
\frac{1}{\vert K_{\lbrace 0,1 \rbrace} \vert} \dfrac{1}{d_1}  \sum_{i=1}^{d_1} \phi (g_0 h_i^1 K_{\lbrace 0 \rbrace} ).
\end{dmath*}
Therefore for every $g_0 K_{\lbrace 1 \rbrace} \in V_1$ and every $\phi \in \ell_2^0 (V_0)$,
\begin{dmath*}
(L_1^{-1} \rho (k_{\lbrace 1 \rbrace}) L_0 \phi) (g_0 K_{\lbrace 1 \rbrace}) = \vert K_{\lbrace 0,1 \rbrace} \vert (\rho (k_{\lbrace 1 \rbrace}) L_0 \phi) (g_0) =
\dfrac{1}{d_1}  \sum_{i=1}^{d_1} \phi (g_0 h_i^1 K_{\lbrace 0 \rbrace} ).
\end{dmath*}
By Proposition \ref{incidence in cosets prop}, we have that $g_0 K_{\lbrace 1 \rbrace}$ is connected by an edge to $g_0 h_i^1 K_{\lbrace 0 \rbrace}$, $i=1,...,d_1$ and therefore our computation yields that $L_1^{-1} \rho (k_{\lbrace 1 \rbrace}) L_0$ is exactly the averaging operator $M_0$ on $\ell_2^0 (V_0)$, i.e., that for every $\phi \in \ell_2^0 (V_0)$,
$$(M_0 \phi) (g_0 K_{\lbrace 1 \rbrace}) = ((L_1^{-1} \rho (k_{\lbrace 1 \rbrace}) L_0) \phi) (g_0 K_{\lbrace 1 \rbrace}),$$
as needed.
\end{proof}

As a corollary, we get a proof of Theorem \ref{orthogonality-spec thm}:
\begin{corollary}
\label{Spectral gap as orthogonality coro}
Let $G$ be a finite group with subgroups $K_{ \lbrace 0 \rbrace}, K_{ \lbrace 1 \rbrace}$, such that $\langle K_{ \lbrace 0 \rbrace}, K_{ \lbrace 1 \rbrace} \rangle =G$. For every $0 \leq \varepsilon$, the second largest eigenvalue $\lambda$ of $X(G, ( K_{\lbrace 0 \rbrace}, K_{\lbrace 1 \rbrace}))$ fulfills $\lambda \leq \varepsilon$ if and only if $K_{ \lbrace 0 \rbrace}$ and $K_{ \lbrace 1 \rbrace}$ are $\varepsilon$-orthogonal.
\end{corollary}

\begin{proof}
Combine Theorem \ref{Spectral gap as orthogonality thm} with Proposition \ref{connection - spec gap and M_0, M_1 norm prop}.
\end{proof}

\begin{remark}
Theorem \ref{orthogonality-spec thm} can be generalized to the case where $G$ is a compact group and $K_{\lbrace 0 \rbrace}, K_{\lbrace 1 \rbrace} < G$ are of finite index (and generate $G$). In that case, we define 
$$k_{\lbrace i \rbrace} = \frac{\mathbbm{1}_{K_{\lbrace i \rbrace}}}{\mu (K_{\lbrace i \rbrace})} \text{ for } i=0,1 \text{ and } k_{G} = \frac{\mathbbm{1}_{G}}{\mu (G)},$$
where $\mu$ is the Haar measure of $G$. For every unitary representation $(\pi, \mathcal{H})$, the operators $\pi (k_{\lbrace i \rbrace}), \pi (k_{G})$ are defined via the Bochner integral, e.g., for $x \in \mathcal{H}$,
$$\pi (k_{\lbrace 0 \rbrace}). x = \frac{1}{\mu (K_{\lbrace 0 \rbrace})} \int_{K_{\lbrace 0 \rbrace}}  (\pi (g).x) d \mu (g).$$
With these definitions, all the steps in the above proof of Theorem \ref{orthogonality-spec thm} can be repeated virtually verbatim.
\end{remark}

\end{document}